\documentclass[11pt,leqno]{amsart}

\usepackage{amsmath,amsfonts,amssymb,amsthm,yhmath}
\usepackage{graphicx}
\usepackage{mathtools}
\usepackage{subcaption}
\usepackage{float}

\theoremstyle{plain}

\newtheorem*{Def*}{Definition}

\newtheorem{Le}{Lemma}
\newtheorem{Prop}{Proposition}
\newtheorem{PropA}{Proposition}

\newtheorem*{Prop*}{Proposition}

\newtheorem*{Cor*}{Corollary}

\newtheorem*{Rem*}{Remark}
\newtheorem{Rem}{Remark}

\newtheorem{Thm}{Theorem}
\newtheorem*{Thm*}{Theorem}
\newtheorem*{OProb*}{Open Problem}

\newcommand{\Ce}{C_\epsilon}
\newcommand{\woa}{w_{0,1}}
\newcommand{\wob}{w_{0,2}}

\newcommand{\R}{\mathbb R}
\newcommand{\C}{\mathbb C}

\newcommand{\eps}{\epsilon}

\DeclareMathOperator{\diam}{diam}

\DeclareMathOperator{\dist}{dist}

\title{A geometric property of quadrilaterals}
\author[Efstathios-Konstantinos Chrontsios-Garitsis and Aimo Hinkkanen]{Efstathios-Konstantinos Chrontsios-Garitsis$^{1},\, $Aimo Hinkkanen$^{2}$}

\address{$^{1}$ Department of Mathematics, University of Illinois Urbana--Champaign, U.S.A., ekc3@illinois.edu, echronts@gmail.com}
\address{$^{1}$ Current address: Department of Mathematics, University of Tennessee, Knoxville, U.S.A., echronts@utk.edu}

\address{$^{2}$ Department of Mathematics, University of Illinois Urbana--Champaign, U.S.A., aimo@illinois.edu}

\date{}

\begin{document}

\maketitle

\begin{abstract}
Quadrilaterals in the complex plane play a significant part in the theory of planar quasiconformal mappings. Motivated by the geometric definition of quasiconformality, we prove that every quadrilateral  with modulus in an interval $[1/K, K]$, where $K>1$, contains a disk lying in its interior, of radius depending only on the internal distances between the pairs of opposite sides of the quadrilateral and on $K$.
\end{abstract}

\setlength{\abovedisplayskip}{12pt}
\setlength{\belowdisplayskip}{12pt}

\section{Introduction}\label{section:introduction}
A \textbf{quadrilateral} $Q=Q(v_1, v_2, v_3, v_4)$ is a bounded Jordan domain in the complex plane $\C$ with four distinct points, called \textbf{vertices}, selected on the boundary and labeled in counter-clockwise order as $v_1, v_2, v_3, v_4$. Recall that a Jordan domain is an open and connected set whose boundary is a homeomorphic image of the circle $\mathbb{S}^1\subset \C$. Quadrilaterals are essential for the geometric definition and properties of planar quasiconformal mappings. More specifically, given a quadrilateral $Q(v_1,v_2,v_3,v_4)$, it can be mapped under a conformal map $\phi$ onto the rectangle $Rec(Q)\subset \C$ with vertices $(0, M(Q), i+M(Q), i) = (\phi(v_1), \phi(v_2), \phi(v_3), \phi(v_4))$, for some $M(Q)>0$. This uniquely defined number $M(Q)$ is called the \textbf{modulus} of $Q$. A homeomorphism $f:\Omega \rightarrow \Omega'$ between domains in $\C$ is $K$-\textbf{quasiconformal} if there is $K\geq 1$ such that
$$M(f(Q))\leq K M(Q)$$ for all quadrilaterals $Q$ whose closure lies in $\Omega$. 

As a result, there is interest in determining properties that quadrilaterals of uniformly bounded modulus might satisfy. Before stating our main result, we need the notion of internal distances. Given a quadrilateral $Q=Q(v_1, v_2, v_3, v_4)$ we define its $a$-sides to be the two disjoint arcs on its boundary from $v_1$ to $v_2$ and from $v_3$ to  $v_4$ that do not contain any vertices in their interior, and its $b$-sides similarly the arcs from $v_2$ to $v_3$ and from  $v_4$ to $v_1$.  The \textbf{internal distance between the $\boldsymbol{a}$-sides} of $Q$ is defined as
$$
s_a(Q):= \inf \{ \ell(C): C \subset Q \, \text{a Jordan arc with end points on different} \, a\text{-sides} \},
$$ where $\ell(C)$ is the length of $C$. Similarly, we define the \textbf{internal distance between the $\boldsymbol{b}$-sides} of $Q$ as
$$
s_b(Q):= \inf \{ \ell(C): C \subset Q \, \text{a Jordan arc with end points on different} \,b\text{-sides} \}.
$$ 

Our main result is the following geometric property for quadrilaterals.

\begin{Thm}\label{thm_main_Quad}
	For every $K\geq 1$ there is a constant $\delta\in (0,1)$ depending only on $K$ such that  every quadrilateral $Q$ with $M(Q)\in [1/K, K]$ contains a disk of radius $\delta \max \{ s_a(Q), s_b(Q) \} $.
\end{Thm}

\paragraph{\textbf{Acknowledgements.}} This material is based upon work supported by the National Science Foundation under Grant No.~1600650. The authors also wish to thank C. Bishop, P. Koskela and the anonymous referee for their valuable comments.

\section{Notation}
	Let $C$ be a closed Jordan arc, i.e., a homeomorphic image of the interval $[0,1]\subset \R$. We denote the length of $C$ by $\ell (C)\in [0,\infty]$. Moreover, if $z, w \in C$ then we denote by $C(z,w)$ the closed sub-arc of $C$ with end points $z$ and $w$. The same notation for sub-arcs will be used for open Jordan arcs and Jordan curves. The notation $C$ will not be used for constants to avoid confusion.
    
    For $z$, $w\in \C$ we denote by $[z,w]$ and $(z,w)$ the closed and open line segment connecting $z$ to $w$, respectively.
    
	For a given quadrilateral $Q$, we denote by $\partial_aQ$ and $\partial_bQ$ the union of the $a$-sides and the union of the $b$-sides of $Q$ respectively, and set $\partial_{a_1}Q:=\partial Q(v_1,v_2)$, $\partial_{b_1}Q:=\partial Q(v_2,v_3)$, $\partial_{a_2}Q:=\partial Q(v_3,v_4)$, $\partial_{b_2}Q:=\partial Q(v_4,v_1)$. Note that since $\partial Q$ is a Jordan curve, the notation $\partial Q(v_k,v_l)$  denotes the closed sub-arc of $\partial Q$ with end points $v_k$, $v_l$ that does not intersect other vertices.
	
	For $z\in {\mathbb C}$ and $r>0$, we write $D(z,r)= \{ w\in {\mathbb C} \colon |w-z|<r \}$ and $ \overline{D(z,r)}= \{ w\in {\mathbb C} \colon |w-z| \leq r \}$. If $E\subset {\mathbb C}$, we denote the Euclidean diameter of $E$ by $\diam E$. If $E$ and $F$ are non-empty subsets of ${\mathbb C}$, we denote the Euclidean distance between $E$ and $F$ by $\dist (E,F)$.

\section{Preliminary reductions}

\subsection{Avoiding the vertices of $Q$}

Let $Q$ be a quadrilateral. For $\delta>0$ with 
	\begin{equation}\label{original_delta_bounds_diam}
		10 \delta<\min\{ \diam (\partial_{a_1}Q), \diam (\partial_{a_2}Q), \diam (\partial_{b_1}Q), \diam (\partial_{b_2}Q)\}
	\end{equation}
	and
	\begin{equation}\label{original_delta_bounds_dist}
	10 \delta<\min\{ \dist (\partial_{a_1}Q, \partial_{a_2}Q),  \dist (\partial_{b_1}Q, \partial_{b_2}Q)\}  ,
	\end{equation}
	define 
	\begin{equation*}
		\begin{split}
			s_a^\delta(Q):=\inf \{ &\ell(C): C \subset Q \, \text{is a Jordan arc with end points on}\\ &\partial_{a_1}Q\setminus (D(v_1,\delta)\cup D(v_2,\delta)) \, \text{and} \, \partial_{a_2}Q\setminus (D(v_3,\delta)\cup D(v_4,\delta)) \}
		\end{split}
	\end{equation*}
	and
	\begin{equation*}
	\begin{split}
	s_b^\delta(Q):=\inf \{ &\ell(C): C \subset Q \, \text{is a Jordan arc with end points on}\\ &\partial_{b_1}Q\setminus (D(v_2,\delta)\cup D(v_3,\delta)) \, \text{and} \, \partial_{b_2}Q\setminus (D(v_1,\delta)\cup D(v_4,\delta)) \}.
	\end{split}
	\end{equation*}
	Note that the Jordan arcs considered in the definitions of $s_a^\delta(Q)$ and $s_b^\delta(Q)$ may contain points that are very close to the vertices of $Q$, for instance there might be $z\in C$ with $|z-v_1|<\delta$, as long as $z$ is not an end point of $C$.	
\begin{Le}\label{Le_s_a_delta}
	For all quadrilaterals $Q$ and $\delta>0$ satisfying \eqref{original_delta_bounds_diam} and \eqref{original_delta_bounds_dist} we have
	\begin{equation}\label{eq:sa^delta compared}
	s_a(Q)\leq s_a^\delta(Q)\leq s_a(Q)+4\pi \delta
	\end{equation}
	and
	\begin{equation}\label{eq:sb^delta compared}
	s_b(Q)\leq s_b^\delta(Q)\leq s_b(Q)+4\pi \delta  .
	\end{equation}
\end{Le}

\begin{proof}
	Let $Q=Q(v_1,v_2,v_3,v_4)$ be a quadrilateral and suppose that $\delta>0$ satisfies \eqref{original_delta_bounds_diam} and \eqref{original_delta_bounds_dist}. Note that by the definition of $\delta$, the four disks $D(v_j,2\delta)$ for $1\leq j\leq 4$ have disjoint closures whose union does not completely contain any Jordan arc in $Q$ that joins the $a$-sides of $Q$, or that joins the $b$-sides of $Q$.
	
	The left hand sides of inequalities \eqref{eq:sa^delta compared} and \eqref{eq:sb^delta compared} follow from the definitions of the internal distances and those of $s_a^\delta (Q)$ and $s_b^\delta(Q)$.
	
	To prove the right hand inequality (\ref{eq:sa^delta compared}), suppose that $\varepsilon>0$ and that $\gamma$ is a rectifiable Jordan arc in $Q$ with end points $z_1\in \partial_{a_1}Q$ and $z_2\in \partial_{a_2}Q$, and with
$\ell(\gamma) < s_a(Q)  +\varepsilon$.  

If $z_1\notin D(v_1,\delta)\cup D(v_2,\delta)$ and $z_2\notin D(v_3,\delta)\cup D(v_4,\delta)$, then the arc $\gamma$ can be used in the definition of $s_a^\delta(Q)$, so that $s_a^\delta(Q) < s_a(Q)  +\varepsilon$. Otherwise, we modify $\gamma$ close to $z_1$ and/or $z_2$ to get another arc $\gamma'$ that can be used in the definition of $s_a^\delta(Q)$, such that $\ell(\gamma') \leq \ell(\gamma) + 4\pi \delta$, which implies that 
$s_a^\delta(Q) \leq \ell(\gamma')  < s_a(Q)  +\varepsilon + 4\pi \delta$. Since $\varepsilon>0$ is arbitrary, this implies the right hand inequality (\ref{eq:sa^delta compared}). The same method can be used to prove right hand inequality (\ref{eq:sb^delta compared}).

We explain how to modify $\gamma$ close to $z_1$, if necessary, to obtain an arc $\gamma_1$ such that $\ell(\gamma_1) \leq \ell(\gamma) + 2\pi \delta$ and such that the end point, say $z_3$, of $\gamma_1$ on $\partial_{a_1}Q$ is outside $D(v_1,\delta)\cup D(v_2,\delta)$. Performing a similar modification close to $z_2$, if necessary, gives rise to an arc $\gamma'$ with the required properties. 

We will need to modify $\gamma$ close to $z_1$ if  $z_1\in D(v_1,\delta)\cup D(v_2,\delta)$. Since these two disks have disjoint closures, suppose that $z_1\in D(v_1,\delta)$. The argument is similar if $z_1\in D(v_2,\delta)$.

Pick a point $z_0\in \gamma$ such that $z_0 \notin \cup_{j=1}^4 \overline{D(v_j,2\delta)}$.
This is possible as we have observed earlier. 

Choose $\rho>0$ such that $|z_1-v_1|< \rho < \delta$. When following $\gamma$ starting from $z_1$, let $w$ be the first point on $\gamma$ such that $|v_1-w|=\rho$. Then $w\in Q$. Further, the open arc of $\gamma$ from $z_1$ to $w$ lies in $Q$ and also in $D(v_1,\delta)$.

	We will use the following theorem due to Ker\'{e}kj\'{a}rt\'{o} (\cite{N}, p.~172). Let $J_1$ and $J_2$ be Jordan curves in the Riemann sphere $ \overline{ {\mathbb C} } =  {\mathbb C} \cup \{\infty\}$ such that $J_1\cap J_2$ contains at least two points. Then every connected component of $ \overline{ {\mathbb C} } \setminus ( J_1\cup J_2)$ is a Jordan domain. We will apply this theorem twice. For the first application, we take $J_1=\partial Q$ and $J_2=\partial D(v_1,\delta)$. Since each of $\partial_{a_1}Q$ and $\partial_{b_2}Q$  must intersect $J_2$ and not at $v_1$,  the intersection $J_1\cap J_2$ contains at least two points.  Note that both $z_0$ and $w$ lie in $ \overline{ {\mathbb C} } \setminus ( J_1\cup J_2)$. Let $\Omega_0$ be the component of $ \overline{ {\mathbb C} } \setminus ( J_1\cup J_2)$ that contains $z_0$. Then $\Omega_0$ is a Jordan domain and $\Omega_0\subset Q$. Further, $\partial \Omega_0 \subset J_1\cup J_2$. 
	
We next apply Ker\'{e}kj\'{a}rt\'{o}'s theorem to the Jordan curves $J_1=\partial Q$ and $J_3=\partial \Omega_0$. The part of $\gamma$ traced from $z_0$ to $z_1$ has a first point $z'$ that intersects $J_2$, and $z'$ lies on an arc of $J_2$ whose end points lie on $J_1$. These two end points cannot coincide since $J_1\cap J_2$ contains at least two points. This arc of $J_2$ is a subset of $J_3$, so $J_1\cap J_3$ contains at least two points.  We find that every connected component of $ \overline{ {\mathbb C} } \setminus ( J_1\cup J_3)$ is a Jordan domain. Two such components are $\Omega_0$ and $\overline{ {\mathbb C} } \setminus \overline{Q}$. 
All other components are also components of $Q\setminus \overline{ \Omega_0 }$. Let $\Omega_1$ be the component of $ \overline{ {\mathbb C} } \setminus ( J_1\cup J_3)$  containing $w$. We have $\Omega_1 \not= \Omega_0$ since one cannot connect $z_0$ to $w$ without intersecting $J_2$. The open arc of $\gamma$ from $w$ to $z_1$ lies in $\Omega_1$ since it does not intersect either $J_1$ or $J_2$ (note that $\partial \Omega_1 \subset J_1\cup J_2$). 

Recall that $\partial \Omega_0 \subset J_1\cup J_2$. The set $\partial \Omega_0 \setminus \partial Q = \partial \Omega_0 \setminus J_1$ is an open subset of the Jordan curve $\partial \Omega_0$ and hence consists of at most countably many Jordan arcs, each of which is an arc of $J_2=\partial D(v_1,\delta)$.  Now $\partial \Omega_1$ must contain such an arc of $J_2$, say $\gamma''$. Then the end points, say $w_1$ and $w_2$, of $\gamma''$ lie on $\partial Q$, hence $\gamma''$ is a cross cut of $Q$ and divides $Q$ into exactly two components, say $Q_1$ and $Q_2$, each of which is a Jordan domain. Now $w_1$ and $w_2$ divide $\partial Q$ into two Jordan arcs, and if the corresponding closed Jordan arcs are denoted by $J_4$ and $J_5$, then, with proper labeling, $\partial Q_1 = \gamma'' \cup J_4$ and $\partial Q_2 = \gamma'' \cup J_5$. One of $Q_1$ and $Q_2$, say $Q_1$, contains $\Omega_0$, and the other one coincides with $\Omega_1$, so $Q_2=\Omega_1$. Note that $w_j\not= v_1$ for $j=1,2$ since $|w_j-v_1|=\delta>0$.

By the definition of $\delta$, each of the points $w_1$ and $w_2$ (since they lie on $J_2=\partial D(v_1,\delta)$ and on $\partial Q$) can only belong to  $\partial_{a_1}Q$ or $\partial_{b_2}Q$. To get a contradiction, suppose that they both belong to 
$\partial_{b_2}Q$ and hence to the interior of $\partial_{b_2}Q$ since they are different from $v_1$. Then $\partial \Omega_1 \setminus \gamma''$ either is contained in $\partial_{b_2}Q$, or contains $\partial Q\setminus \partial_{b_2}Q$ and in particular contains $z_2$. Since $w\in \Omega_1$, we have $\Omega_1 \subset D(v_1,\delta)$, hence $z_2\notin \partial \Omega_1$. It follows that $\partial \Omega_1 \cap \partial Q\subset \partial_{b_2}Q$. But then $\Omega_1$ cannot contain the arc of $\gamma$ from $w$ to $z_1\in \partial_{a_1}Q$, a contradiction. 

It follows that at least one of $w_1$ and $w_2$, say $w_1$, lies on $\partial_{a_1}Q$.
When we trace $\gamma$ from $z_0$ towards $z_1$, we must enter $\Omega_1$ at some point, hence there will be a first point where we intersect $\partial \Omega_1$. This point, which is the same as the point $z'$ discussed above, is also in $Q$, hence on $\gamma''$. We form the arc $\gamma_1$ by following $\gamma$ from $z_2$ through $z_0$ to $z'$ and then along the arc $\gamma''$ from $z'$ to $w_1$. Then $\gamma_1$ joins the $a-$sides of $Q$ in $Q$ and $\ell(\gamma_1)\leq \ell(\gamma)+2\pi \delta$, as required. 

This completes the proof of Lemma~\ref{Le_s_a_delta}. 
	
\end{proof}

\subsection{A consequence of Rengel's inequality}

Lehto and Virtanen (\cite{LV1}, see also \cite{LV}, Lemma 4.1) obtained  the following consequence of Rengel's inequality (\cite{LV}, p.~22).

\begin{PropA} \label{Quad_Prop:A}
	The modulus of a quadrilateral $Q$ satisfies the inequality
	$$
	\frac{  ( \log(1+2s_b(Q)/s_a(Q)) )^2}{\pi+2\pi \log(1+2s_b(Q)/s_a(Q))} \leq M(Q)\leq  \frac{\pi+2\pi \log(1+2s_a(Q)/s_b(Q))}{  (  \log(1+2s_a(Q)/s_b(Q)) )^2}.
	$$
\end{PropA}

Due to Proposition \ref{Quad_Prop:A}, we can replace the condition on the modulus in Theorem \ref{thm_main_Quad} by the equivalent condition that the ratio of internal distances lies in a bounded interval. More specifically, if $\mathcal{Q}_M(K)$ is the collection of all quadrilaterals $Q$ with $M(Q)\in [1/K,K]$, then by Proposition A there exists $\tilde{L}>1$ depending only on $K$ such that the collection of all quadrilaterals $\tilde{Q}$ with ratio $s_a(\tilde{Q})/s_b(\tilde{Q}) \in [1/\tilde{L},\tilde{L}]$, denoted by $\mathcal{Q}(\tilde{L})$, contains $\mathcal{Q}_M(K)$. Similarly, for $\tilde{L}>1$ there is $\tilde{K}>1$ larger than $K$, for which $\mathcal{Q}(\tilde{L})\subset\mathcal{Q}_M(\tilde{K})$. We will thus prove the assertion of Theorem \ref{thm_main_Quad} for all $Q\in \mathcal{Q}(\tilde{L})$, for fixed $\tilde{L}>1$ and for $\delta\in (0,1)$ depending on $\tilde{L}$.

\begin{Rem}
	A sharper form of the inequality in Proposition \ref{Quad_Prop:A} was later proved by Hanson and Herron in \cite{HH}, which could be useful in obtaining a sharper constant $\delta$ in Theorem \ref{thm_main_Quad}.
\end{Rem}

\subsection{Approximation of quadrilaterals}

Lemma~\ref{le:quad_approx} below allows us to only consider quadrilaterals $Q\in \mathcal{Q}(L)$ that have boundary consisting of finitely many line segments, each line segment being parallel to one of the coordinate axes, where $L=3\tilde{L}$. Denote the collection of all such quadrilaterals by $\mathcal{Q}_{\text{ls}}(L)$.

\begin{Le}\label{le:quad_approx}
	For every $Q\in \mathcal{Q}(\tilde{L})$ and every $\tau\in (0,1/2]$ there is a quadrilateral $Q_\tau\in \mathcal{Q}_{\text{ls}}(L_\tau)$ contained in $Q$ with $|s_a(Q_\tau)-s_a(Q)|\leq\tau \min\{s_a(Q),s_b(Q)\}$ and $|s_b(Q_\tau)-s_b(Q)|\leq\tau\min\{s_a(Q),s_b(Q)\}$, where $L_\tau= \frac{1+\tau}{1-\tau}\tilde{L}\leq L$.
\end{Le}

\begin{proof}
	Let $Q=Q(v_1,v_2,v_3,v_4)\in \mathcal{Q}(\tilde{L}) $, $\rho (z) = \sup \{ \rho>0: D(z,\rho)\subset Q \}$ for all $z\in Q$, $\rho_0=\sup\{ \rho(z): z\in Q \}$ and consider a closed disk $D=  \overline{D(z_0,\rho_0/2)}$ that lies in $Q$. We can map the quadrilateral $Q$ onto a rectangle $Rec(Q)$ using a conformal map $\phi$ so that $\phi (\partial_{a_1}Q)=(0,M)$, $\phi (\partial_{b_1}Q)=(M,M+i)$,  $\phi (\partial_{a_2}Q)=(i,M+i)$, $\phi (\partial_{b_2}Q)=(0,i)$, where $M=\text{Mod}(Q)$. Let $l_j$ be the line segment in the closure of $Rec(Q)$ connecting $\phi(z_0)$ to $\phi(v_j)$, for all $j\in \{1,2,3,4\}$. Let $C_j= \phi^{-1}(l_j)$ be Jordan arcs that lie in $Q$ except for their end points and connect $z_0$ to $v_j$, for all $j\in \{1,2,3,4\}$.
	
	By Lemma~\ref{Le_s_a_delta}, for each $\eps>0$ there are some positive $\delta_\eps<\epsilon$ and Jordan arcs $C_{a\eps}$, $C_{b\eps}$ connecting the $a$-sides and $b$-sides of $Q$, respectively, both with end points outside $\bigcup_{j=1}^4 \overline{D(v_j, \delta_\eps)}$ and with $\ell(C_{a\eps})\leq s_a(Q)+\epsilon$ and $\ell(C_{b\eps}) \leq s_b(Q)+\epsilon$. Fix positive $\epsilon < 10^{-3}\min\{s_a(Q),s_b(Q)\}$ and denote by $z_{a_1}\in \partial_{a_1}Q$, $z_{a_2}\in \partial_{a_2}Q$ the end points of $C_{a\eps}$ and by $z_{b_1}\in \partial_{b_1}Q$, $z_{b_2}\in \partial_{b_2}Q$ the end points of $C_{b\eps}$. Fix a positive $\delta<\rho_0/10$ satisfying all of the following inequalities:
	\begin{equation}\label{eq: approx_delta_z_a_z_b}
	\delta <\frac{\min\{|z_{a_i}-z_{b_j}|: i=1, 2, \, \, j=1, 2 \}}{100}  , 
	\end{equation}
	\begin{equation}\label{eq: approx_delta_z_a_vertices}
	\delta <\frac{\min\{|z_{a_i}-v_j|: i=1, 2, \, \, j=1, 2, 3, 4 \}}{100}  ,  
	\end{equation}
	\begin{equation}\label{eq: approx_delta_z_b_vertices}
	\delta <\frac{\min\{|z_{b_i}-v_j|: i=1, 2, \, \, j=1, 2, 3, 4 \}}{100}  ,  
	\end{equation}
	\begin{equation}\label{eq: approx_delta_sides_and_diam}
	\delta <\frac{\min\{\diam (\partial_{a_1}Q), \diam(\partial_{a_2}Q), \diam(\partial_{b_1}Q), \diam(\partial_{b_2}Q) \}}{100}  ,  
	\end{equation}
	\begin{equation}\label{eq: approx_delta_distances_of_sides}
	\delta <\frac{\min\{ \dist(\partial_{a_1}Q, \partial_{a_2}Q), \dist(\partial_{b_1}Q, \partial_{b_2}Q), \delta_\eps\}}{100}  .  
	\end{equation}
	
	Denote by $C_{a\eps\delta}$ the closed sub-arc of $C_{a\eps}$ from $\tilde{z}_{a_1}$ to $\tilde{z}_{a_2}$, where $\tilde{z}_{a_j}$ is the point such that $\ell(C_{a\eps}(\tilde{z}_{a_j}, z_{a_j}))=\delta$ for $j=1, 2$, and similarly denote by $C_{b\eps\delta}$ the closed sub-arc of $C_{b\eps}$ from $\tilde{z}_{b_1}$ to $\tilde{z}_{b_2}$, where $\tilde{z}_{b_j}$ is the point such that $\ell(C_{b\eps}(\tilde{z}_{b_j}, z_{b_j}))=\delta$ for $j=1, 2$.
	
	Moreover, for all $j\in \{1,2,3,4\}$, denote by $C_j^\delta$ the closed Jordan arc $C_j(z_0, \tilde{v}_j)$, where $\tilde{v}_j= r_j(t_j)$ for 
	$$t_j=\max \{ t\in [0,1]: r_j(t)\in \partial D(v_j, \delta), \, r_j((t,1])\subset D(v_j,\delta) \}$$
	and $r_j:[0,1] \rightarrow \C$ is a homeomorphism onto $C_j$ with $r_j(0)=z_0$. In other words, $\tilde{v_j}$ is the ``last" point of $C_j$ intersecting the boundary of $D(v_j,\delta)$ with the direction on $C_j$ from $z_0$ towards $v_j$, after which $C_j$ lies in $D(v_j,\delta)$.
	
	Similarly to how we defined $C_j$, $j=1,2,3,4$, we can find Jordan arcs $C_{0,a}$ and $C_{0,b}$ inside $Q$ that connect $z_0$ to $C_{a\eps\delta}$ and $C_{b\eps\delta}$, respectively. Set
	\begin{equation*}
	K= \left( \bigcup_{j=1}^4 C_j^\delta \right) \cup D \cup C_{a\eps\delta} \cup C_{b \eps \delta} \cup C_{0,a} \cup C_{0,b},
	\end{equation*}
	which is schematically depicted in Figure \ref{fig_connected_K} (it is not intended that the curves shown would be the actual curves obtained, among other things, by applying a conformal mapping to the quadrilateral shown in Figure \ref{fig_connected_K}).
	
	\begin{figure}[H]
		\centering
		\includegraphics[width=0.85\textwidth]{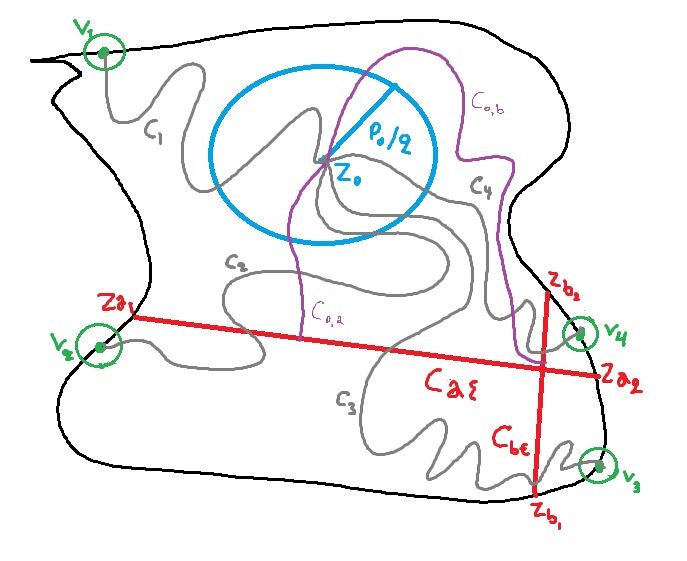}
		\caption{An example of a set $K$ constructed inside a quadrilateral $Q$.}
		\label{fig_connected_K}
	\end{figure}

	Similarly to how the points $\tilde{v}_j$ were defined, we set $a_{v_1}$ to be the unique point of $\partial_{a_1} Q$ with direction from $v_2$ towards $v_1$ that intersects the boundary of $D(v_1, \delta)$ such that the arc $\partial_{a_1} Q(a_{v_1},v_1)$ stays in $\overline{D(v_1, \delta)}$, and $b_{v_1}$ to be the unique point of $\partial_{b_2} Q$ with direction from $v_4$ towards $v_1$ that intersects the boundary of $D(v_1, \delta)$ such that the arc $\partial_{b_2} Q(b_{v_1},v_1)$ stays in $\overline{D(v_1, \delta)}$. In the same way we can define $a_{v_j}$, $b_{v_j}$ for all $j\in \{1,2,3,4\}$. Note that \eqref{eq: approx_delta_sides_and_diam} ensures that such points exist and \eqref{eq: approx_delta_distances_of_sides} guarantees that for all $j \in \{1,2,3,4\}$ the only sides of the boundary of $Q$ intersecting $D(v_j, 2 \delta)$ are the ones that meet at $v_j$.
	
	Let $\tilde{C}_j := C_j(\tilde{v_j},v_j)$ be the closed sub-arc of $C_j$ from $\tilde{v}_j$ to $v_j$ for all $j\in \{1,2,3,4\}$ and fix positive $d_1$ and $d_2$ so that for all $j\in \{1,2,3,4\}$ we have
	$$
	d_1 <  \dist (\partial Q(a_{v_j},b_{v_j})\cup \tilde{C}_j, 
	\overline{  (\partial Q \setminus \partial Q(a_{v_j},b_{v_j} )    )    
	\cap  D(v_j, \delta )  }  )  ,
	$$ where $\partial Q(a_{v_j},b_{v_j})$ denotes the closed sub-arc of $\partial Q$ passing through $a_{v_j}$, $b_{v_j}$ and $v_j$, and
	$$
	d_2 < \dist(\partial_{a'} Q, \partial_{b'} Q ),
	$$ where $\partial_{a'} Q:= \overline{\partial_a Q \setminus \bigcup_{j=1}^{4} \partial Q(a_{v_j}, b_{v_j})}$ and $\partial_{b'} Q:=\overline{\partial_b Q \setminus \bigcup_{j=1}^{4} \partial Q(a_{v_j}, b_{v_j})}$. Note that by definition of $a_{v_j}$ and $b_{v_j}$ all distances defined above are positive, even if $\partial Q$ were to contain an arc of $\partial D(v_j,\delta) $.
	
	We are now ready to start the approximation of $\partial Q$ by finitely many line segments. Fix some $s>0$ such that
	\begin{equation}\label{eq: side-length}
		s< \frac{\min\{ \dist(K, \partial Q), d_1, d_2, \delta \}}{100}  	\end{equation} and cover the plane with closed axes oriented squares of side length $s$, i.e., with sides parallel to the coordinate axes that have length $s$.

	Note that squares that do not intersect the closure of $\bigcup_{j=1}^4 D(v_j,\delta)$ and contain points of one of the sides of $Q$,  cannot contain points of any of the other three sides, because their diameter is much smaller than $d_2$ and $\delta$ (see \eqref{eq: approx_delta_sides_and_diam}).

Let $S_0$ be an arbitrary closed square like this that intersects $\partial Q$. Then there is a point $\zeta_1\in S_0\cap \partial Q$. If there is also a point $\zeta_2\in  S_0\cap K$, then
$$
\dist(K, \partial Q) \leq |\zeta_1-\zeta_2| \leq \diam S_0 = s\sqrt{2} ,
$$
which contradicts \eqref{eq: side-length}. 
Let $S$ be the union of those closed squares that intersect $\partial Q$. Hence
$$
\partial Q \subset S , \quad S\cap K=\emptyset .
$$
Since $S$ is a compact subset of ${\mathbb C}$, its complement ${\mathbb C}\setminus S$ is the union of open connected components, and each bounded component is a subset of $Q$. One of the bounded components contains the connected set $K$. We denote this component by $Q_{\tau}$. The boundary of 
$Q_\tau$  consists of finitely many line segments, each parallel to one of the coordinate axes. 

We next prove that $Q_\tau$ is a Jordan domain. To show that, it is enough to show there are no self-intersections on $\partial Q_\tau$. Assume towards a contradiction that there are self-intersections on $\partial Q_\tau$. Since $\partial Q_\tau$ consists only of sides of axes oriented squares, the only self-intersections that could occur are due to two of said squares intersecting each other at a corner, say $v_*$, for which there is an open disk $D(v_*,r_*)$ of a tiny radius $r_*\in (0, s/2)$ such that it only intersects the two aforementioned squares and no other squares intersecting $\partial Q_{\tau}$, and the domain $Q_\tau$. Let $q_1$ be a point in one of the connected components of $Q_\tau\cap D(v_*,r_*)$ and $q_2$ be a point in the other component. Since $Q_\tau$ is connected, there is a path $C_q \subset Q_\tau$ that connects $q_1$ to $q_2$ and intersects no point of $\partial Q_\tau$, and, as a result, no point of $\partial Q$. Note that by the choice of the squares contained in $S$, the corner $v_*$ cannot lie on the boundary of $Q$, since then all four squares that meet at $v_*$ would lie in $S$ and the boundary of $Q_\tau$ would have no self-intersection at $v_*$. Hence, the Jordan curve $\partial Q$ has points lying in the bounded domain bounded by $[q_1, v_*]\cup [v_*,q_2]\cup C_q$ and cannot intersect its boundary. But then it would be impossible to have points of $\partial Q$ in the square lying in the complement of the aforementioned domain and intersecting it at $v_*$, since $\partial Q$ is connected, leading to a contradiction. See Figure \ref{fig_no_self_intersect}.
	
	\begin{figure}[H]
		\centering
		\includegraphics[width=0.7\textwidth]{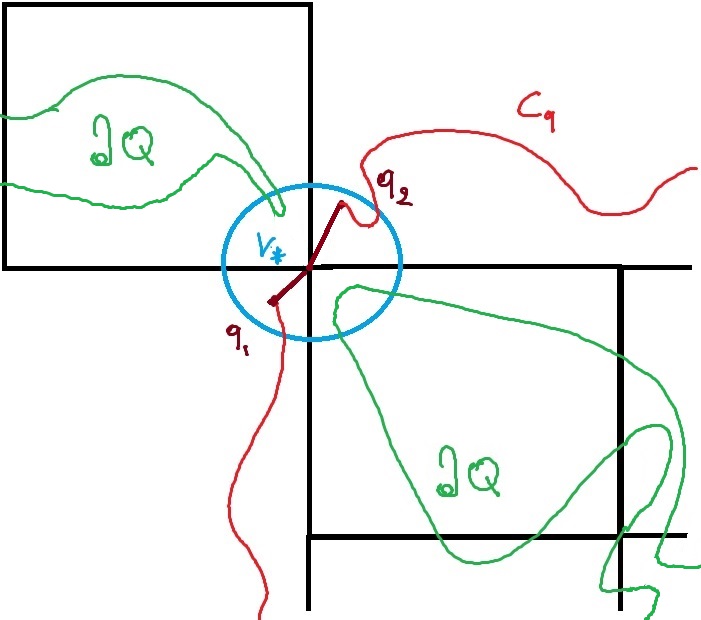}
		\caption{An illustration of what would happen if we assumed $\partial Q_\tau$ is not a Jordan curve.}
		\label{fig_no_self_intersect}
	\end{figure}

	Hence, $Q_\tau$ is a Jordan domain. We choose four vertices $v_j'$ for $1\leq j\leq 4$ on $\partial Q_\tau$ to obtain a quadrilateral $Q_\tau(v_1',v_2',v_3',v_4')$   that lies in $Q$, with boundary consisting of finitely many line segments. We choose the points $v_j'$ as follows.

	Note that squares that intersect $\partial Q(a_{v_1},b_{v_1})\cup \tilde{C}_1$
	cannot intersect points of $(\partial Q \setminus \partial Q(a_{v_1},b_{v_1}))\cap D(v_1, \delta)$ due to \eqref{eq: side-length}. Similarly, squares that intersect $\partial_{a'}Q\cup \partial_{b'} Q$ cannot contain points of $C_1$. So the only squares containing boundary points of $Q$ that may intersect $C_1$ are those containing points of $\partial Q (a_{v_1},b_{v_1})$. More specifically, they may only intersect points of $\tilde{C}_1=C_1(\tilde{v}_1,v_1)$ because the side lengths of the squares are also smaller than $\dist(C_1^\delta, \partial Q)/100$ by definition of $K$ and \eqref{eq: side-length}. Set $v_1'$ to be the first point of $S$ that $\tilde{C}_1$ intersects with direction from $\tilde{v}_1$ towards $v_1$. Similarly the other three $v_j'$ for $j\in \{ 2,3,4\}$ can be defined.
	
	Set $C_j' = C_j(z_0,v_j')$ to be the open sub-arc of $C_j$ from $z_0$ to $v_j'$ for all $j\in \{ 1, 2, 3,4\}$. Now $Q_\tau$  contains the connected set $ K \cup \left( \bigcup_{j=1}^4 C_j' \right )$.

Note that $Q\setminus Q_{\tau}$ may be large, and there may be points on $\partial Q$ that have a large distance to the set $\partial Q_{\tau}$. However, we shall show that every point of $\partial Q_{\tau}$ is close to some point of $\partial Q$. 
Indeed, outside small disks centered at the old vertices $v_j$ each point on a side of $\partial Q_{\tau}$ is close to the corresponding side of $\partial Q$.

	We will now prove that away from the disks $D(v_j,2\delta)$, we can connect within $Q$ boundary points of $Q_\tau$ to boundary points of the corresponding side of $Q$. Let $a_{11}$ be the last point of $\partial_{a_1}Q_\tau$ with direction from $v_1'$ towards $v_2'$ that intersects $\partial D(v_1,2\delta)$ so that the open sub-arc $\partial_{a_1} Q_\tau (a_{11},v_2')$ of $\partial_{a_1} Q_\tau$ does not intersect $\overline{D(v_1,2\delta)}$. Similarly, let $a_{12}$ be the last point of $\partial_{a_1}Q_\tau$ with direction from $v_2'$ towards $v_1'$ that intersects $\partial D(v_2,2\delta)$ so that the open sub-arc $\partial_{a_1} Q_\tau (v_1', a_{12})$ of $\partial_{a_1} Q_\tau$ does not intersect $\overline{D(v_2,2\delta)}$.

	Let $z\in \partial_{a_1}Q_\tau (a_{11},a_{12})$. Then $z$ lies in a closed square $S_z\subset S$. Indeed, at least one of the four boundary segments of $S_z$ is contained in $\partial Q_{\tau}$. By definition, $S_z\cap \partial Q\not= \emptyset$.
Let $T'$ be the last such segment  in $\partial S_z$ when moving along $ \partial_{a_1}Q_\tau (a_{11},a_{12})$ from $a_{11}$ towards $a_{12}$ in case there is more than one such segment. Then there is the next segment $T''$, after $T'$ in the same direction, on $\partial Q_{\tau}$ that is contained in the boundary of another closed square $S'$ in $S$ and is not contained in the boundary of $S_z$. The intersection of the closed squares $S_z$ and $S'$ is non-empty. Suppose that $\zeta_1\in S_z\cap \partial Q$ and $\zeta_2\in S'\cap \partial Q$. It is already clear by the definition of $d_2$ and \eqref{eq: side-length} that $S_z$ cannot contain points from two different sides of $Q$, and the same applies to $S'$. Now 
$|\zeta_1-\zeta_2|\leq 2\sqrt{2} s$. By \eqref{eq: side-length} and the definitions of $d_1$ and $d_2$, the points $	\zeta_1$ and $\zeta_2$ must belong to the same side of $Q$. Thus there is a unique side of $Q$ that is followed by the sub-arc $ \partial_{a_1}Q_\tau (a_{11},a_{12})$ of $ \partial_{a_1}Q_\tau $ all the way from $a_{11}$ to $a_{12}$, and this sub-arc stays bounded away by a definite distance from all other sides of $\partial Q$. (Note that this relation does not go in the other direction: this side of $\partial Q$ may contain points that are far away from every point of $\partial Q_{\tau}$.)   

We claim that this unique side of $\partial Q$, say $\gamma$, is $ \partial_{a_1}Q$. 
Now $\gamma$ contains a point $\zeta_0$ in a closed square that also contains $a_{11}$, and hence $ |  \zeta_0-v_1	 | \leq 2\delta + s\sqrt{2}$. By \eqref{eq: approx_delta_distances_of_sides} and \eqref{eq: side-length}, and since $v_1$ belongs to the closure of each of   $ \partial_{a_1}Q$ and $ \partial_{b_2}Q$, neither   $ \partial_{a_2}Q$ nor $ \partial_{b_1}Q$  can contain any point that close to $v_1$. Hence $\gamma$ can only be $ \partial_{b_2}Q$ nor $ \partial_{a_1}Q$. Similarly close to $a_{12}$, the side $\gamma$ contains a point $\zeta_0'$ with $ |  \zeta_0'-v_2	 | \leq 2\delta + s\sqrt{2}$. By \eqref{eq: approx_delta_distances_of_sides} and \eqref{eq: side-length} $ \partial_{b_2}Q$  cannot contain any point that close to $v_2$. Hence $\gamma$ must be $ \partial_{a_1}Q$.

	The argument runs similarly for points on other sides of $\partial Q_\tau$ that lie on parts of the sides that have exited the corresponding disks $D(v_j,2\delta)$ and do not intersect them again. This shows that those sub-arcs of the sides of $\partial Q_\tau$ are close to the corresponding sides of $\partial Q$ and only those sides of $\partial Q$. 
	
	We will first show that $s_a(Q_\tau)\leq s_a(Q)+\epsilon$, and the proof for $s_b(Q_\tau)\leq s_b(Q)+\epsilon$ follows similarly. Consider the definition of $C_{a\epsilon\delta}\subset Q$ and the fact that the end points of $C_{a\epsilon}$ lie outside $\bigcup_{j=1}^4 D(v_j,2\delta)$ due to \eqref{eq: approx_delta_z_a_vertices}. At each end of $C_{a\epsilon\delta}\subset Q$ we continue outwards along $C_{a\epsilon}$ until we come to the first point that lies on $\partial Q_{\tau}$. This gives rise to an arc $C$ of length $\leq s_a(Q)+\epsilon$ joining two points of $\partial Q_{\tau}$. Now we only need to show that these two points lie on the two $a$-sides of $Q_{\tau}$. It suffices to give the argument for one side and the case of the other side is similar. Let $\zeta_3=\tilde{z}_{a_1}$ be the end point of $C_{a\epsilon\delta}$ after which the remaining part of $C_{a\epsilon\delta}$ when going towards    $ \partial_{a_1}Q$ has length $\delta$. Let $\zeta_4\in \partial_{a_1}Q$ be the end point of $C_{a\epsilon}$.  Recall that $|\zeta_4-v_j|\geq \delta_{\epsilon}>100 \delta$ for all $j$ with $1\leq j\leq 4$. Hence the arc of $C_{a\epsilon}$  from $\zeta_4$ to $\zeta_3$ lies outside $D(v_j,99\delta)$ for $1\leq j\leq 4$. 
	
	Let $\zeta_5$ be the first point on $\partial Q_\tau$ that we encounter when moving towards $ \partial_{a_1}Q$ from $\zeta_3$ along $C_{a\epsilon}$. Then $\zeta_5$ lies on a side $\gamma$ of $\partial Q_\tau$, and by the definition of $S$, there is a point $\zeta_6\in \partial Q$ in a closed square $S_1$ in $S$ with $\zeta_5\in S_1$  such that $|\zeta_5-\zeta_6| \leq s\sqrt{2}$.
Now 	$|\zeta_5-v_j| > 99 \delta$ so that by what we have proved above, there is a unique side $\gamma'$ of $\partial Q$ associated with $\gamma$, and $\zeta_6\in \gamma'$. The point $\zeta_4$ lies on $ \partial_{a_1}Q$ and has distance $<  \delta + s\sqrt{2}$ from $\zeta_6$. Hence, by the definitions of $\delta$ and $s$, we have $\gamma' =  \partial_{a_1}Q$, and consequently $\gamma = \partial_{a_1}Q_{\tau}$, as desired.

	On the other hand, let $C_{a,\epsilon,\tau}$ be a Jordan arc that connects the $a$-sides of $Q_\tau$, lies in $Q_\tau$ except for its end points, which lie in the complement of $\bigcup_{j=1}^4 D(v_j',4\delta)$, and has length at most $s_a^{4\delta}(Q_\tau)+\epsilon$. But by the definition of $Q_\tau$ and the fact that the end points of $C_{a,\epsilon,\tau}$ lie outside $\bigcup_{j=1}^4 D(v_j,2\delta)$, there are line segments inside $Q$ of length at most $\sqrt{2}s$ connecting the end points of $C_{a,\epsilon,\tau}$ to points on $\partial Q$ and hence, by the argument already given, to the $a$-sides of $Q$. Thus, $s_a(Q)\leq s_a^{4\delta}(Q_\tau)+\epsilon +2\sqrt{2}s$. By the definition of $Q_\tau$ and \eqref{eq: approx_delta_sides_and_diam}, \eqref{eq: approx_delta_distances_of_sides}, the constant $4\delta>0$ satisfies \eqref{original_delta_bounds_diam} and \eqref{original_delta_bounds_dist}, so applying Lemma \ref{Le_s_a_delta} to $Q_\tau$ we get that $s_a(Q)\leq s_a(Q_\tau)+\epsilon +2\sqrt{2}s+16 \pi \delta$, which by the choice of $s$ implies that $s_a(Q) \leq s_a(Q_\tau) +\epsilon +67 \delta $. Similarly, the inequality $s_b(Q) \leq s_b(Q_\tau) +\epsilon +67 \delta$ also holds. Thus, we have shown that for all $\epsilon < 10^{-3} \min \{ s_a(Q), s_b(Q)\}$ we have	
	\begin{equation}\label{eq: comparable_epsilon a-distances}
	|s_a(Q)-s_a(Q_\tau)|\leq \epsilon +67 \delta<2\epsilon,
	\end{equation}
	and
	\begin{equation}\label{eq: comparable_epsilon b-distances}
	|s_b(Q)-s_b(Q_\tau)|\leq\epsilon +67 \delta<2\epsilon.
	\end{equation}

	Let $\tau \in (0,1/2]$. Applying \eqref{eq: comparable_epsilon a-distances} and \eqref{eq: comparable_epsilon b-distances} for $\epsilon = \tau \min\{s_a(Q),s_b(Q)\}/2$ we get
	
	\begin{equation}\label{eq: comparable a-distances}
	|s_a(Q)-s_a(Q_\tau)|\leq \tau \min\{s_a(Q),s_b(Q)\}\leq \tau s_a(Q),
	\end{equation}
	and
	\begin{equation}\label{eq: comparable b-distances}
	|s_b(Q)-s_b(Q_\tau)|\leq \tau \min\{s_a(Q),s_b(Q)\}\leq \tau s_b(Q).
	\end{equation}

	It is now clear by \eqref{eq: comparable a-distances} and \eqref{eq: comparable b-distances} that
	$$
\frac{(1-\tau)s_a(Q)}{(1+\tau)s_b(Q)} \leq 	\frac{s_a(Q_\tau)}{s_b(Q_\tau)}\leq \frac{(1+\tau)s_a(Q)}{(1-\tau)s_b(Q)},
	$$ but since $Q\in \mathcal{Q}(\tilde{L})$ we get
	$$
\frac{1-\tau}{1+\tau}\tilde{L} \leq	\frac{s_a(Q_\tau)}{s_b(Q_\tau)}\leq \frac{1+\tau}{1-\tau}\tilde{L}.
	$$ Hence, for $L_\tau=\frac{1+\tau}{1-\tau}\tilde{L}$  and because $\tau\leq 1/2$ we have
	$$
	Q_\tau \in \mathcal{Q}_{\text{ls}}(L_\tau) \subset \mathcal{Q}_{\text{ls}}(3\tilde{L})
	$$ as needed.
	
\end{proof}

\begin{Rem}
	Note that Lemma \ref{le:quad_approx} is particularly useful for $\tau$ very close to $0$. In that case it practically guarantees that the ``approximation" quadrilateral $Q_\tau\in \mathcal{Q}_{\text{ls}}(L_\tau)$ has in fact internal distances very close to those of the original quadrilateral $Q$. The reason we decided to state Lemma \ref{le:quad_approx} in its current form is to point out that all the approximations are contained in the collection $\mathcal{Q}_{ls}(L)$ for $L$ independent of $\tau$.
\end{Rem}

Thus, by Lemma \ref{le:quad_approx}, for $\tau$ very close to $0$, if we prove that a disk of radius $\delta' \max\{s_a(Q_\tau), s_b(Q_\tau)\}$ lies in $Q_\tau\in \mathcal{Q}_{\text{ls}}(L)$ for some $\delta'\in (0,1)$ depending only on $\tilde{L}$ and $L=3\tilde{L}$, then a disk of radius $\delta \max\{s_a(Q), s_b(Q)\}$ lies in $Q$ with $\delta=\delta' /4$, which proves Theorem \ref{thm_main_Quad}.

With the above reduction in mind, it is enough to prove the following:

\begin{Thm}\label{thm_sa}
	For every quadrilateral $Q$ in $\mathcal{Q}_{\text{ls}}(L)$ there is a disk of radius $r:= \frac{s_a(Q)}{1000L}$ that lies inside $Q$.
\end{Thm}

Indeed, suppose $Q(v_1,v_2,v_3,v_4) \in \mathcal{Q}_{\text{ls}}(L)$. Then $Q'=Q(v_2,v_3,v_4,v_1)$ also lies in $\mathcal{Q}_{\text{ls}}(L)$, since $\partial_{a_1}Q'=\partial_{b_1} Q$, $\partial_{a_2}Q' = \partial_{b_2} Q$, $\partial_{b_1} Q' = \partial_{a_1}Q$ and $\partial_{b_2} Q'= \partial_{a_2} Q$, which implies $s_a(Q')=s_b(Q)$ and $s_b(Q')=s_a(Q)$. Applying Theorem \ref{thm_sa} to $Q$ and $Q'$ along with Proposition~A proves Theorem \ref{thm_main_Quad}.

Note that the opposite implication might not be true, i.e., Theorem \ref{thm_main_Quad} does not necessarily imply Theorem \ref{thm_sa}, since the radius of the disk contained in $Q$ might not have exactly the form $r= \frac{s_a(Q)}{1000L}$. However, by Proposition \ref{Quad_Prop:A} and the discussion afterwards, Theorem \ref{thm_main_Quad} is in fact equivalent to the following:
\begin{Thm}
	For every $L\geq 1$ there is a constant $\delta\in (0,1)$ depending only on $L$ such that every quadrilateral Q in $\mathcal{Q}_{\text{ls}}(L)$ contains a disk of radius $\delta \max\{ s_a(Q),s_b(Q) \}$.
\end{Thm}

\section{Proof of Theorem \ref{thm_sa}}
	
	Let $Q=Q(v_1,v_2,v_3,v_4) \in \mathcal{Q}_{\text{ls}}(L)$ and set $s_a:= s_a(Q)$, $s_b:= s_b(Q)$. Note that one can find a Jordan arc that is the union of finitely many line segments, connects the $a$-sides of $Q$ with length $s_a$ and lies in the closure of $Q$. The reason why such an arc exists lies in the fact that $\overline{Q}$ is a connected union of finitely many closed squares with sides parallel to the axes (due to Lemma \ref{le:quad_approx}). Every Jordan arc connecting the $a$-sides of $Q$ inside the interior of $Q$ intersects at most finitely many squares, which forms a chain of squares connecting one side to the other. There is a minimal number of squares needed to perform such a connection. Note that in each such square, the part of the arc lying inside can be made shorter by connecting  with a line segment the first point from where the arc enters the square with the last point from where the arc exits. This line segment may either lie in the interior of the square, or on its boundary. Thus, by definition of $s_a$ and the finiteness of the number of minimal square-chains connecting the $a$-sides, an arc lying in $\overline{Q}$ that connects the $a$-sides while being of the shortest length possible can indeed be found (as a union of sides of the squares in $Q$ and line segments lying in said squares).
	
	Define $r$ as in the statement of Theorem~\ref{thm_sa}. 
	We fix the following:
	\begin{itemize}
		\item Let $C_a \subset \overline{Q}$ be a Jordan arc that is the union of finitely many line segments and connects the $a$-sides of $Q$ with $\ell(C_a)=s_a$. 
		\item For small $\epsilon, \epsilon'>0$ with $\epsilon<10^{-3}r$ let $\Ce := \bigcup_{i=1}^N [z_i,z_{i+1}]$ be a finite union of line segments that connects the $a$-sides of $Q$ with $\ell(\Ce)\leq s_a+\epsilon$ and $$\Ce \subset \bigcup\limits_{\tilde{w}\in C_a}D(\tilde{w},\epsilon').$$ We essentially ``modify" $C_a$ within an $\epsilon'$-neighborhood to get the arc $C_{\epsilon}$, which lies in the interior of $Q$ (except its end points) and has length at most $s_a+\epsilon$. Note that the arcs $C_a$ and $C_{\eps}$ can differ in general (see Figure \ref{RefFig}).
		
		In addition, we can choose $\Ce$ such that $\Ce \setminus \{z_1, z_{N+1}\} \subset Q$ and $z_1\in \partial_{a_1}Q \setminus \partial_bQ$, $z_{N+1}\in \partial_{a_2}Q \setminus \partial_bQ$.
		\item Write $R:=10r=\frac{s_a(Q)}{100L}$, 
		\begin{equation}\label{farendpnts}
			F:= \{ w\in \Ce: \, \min\{\ell(\Ce(z_1,w)), \ell(\Ce(w,z_{N+1})) \}\geq 15R \},
		\end{equation}
		and
		\begin{equation}\label{furtherendpnts}
		F':= \{ w\in \Ce: \, \min\{\ell(\Ce(z_1,w)), \ell(\Ce(w,z_{N+1}))\}\geq 16R+2\epsilon \},
		\end{equation} so that $F$ and $F'$ are two sets of points of $\Ce$ that are sufficiently far from the end points of $\Ce$.
	\end{itemize}
	
	It is not difficult to see that if a line segment $(x,y) \subset Q$ intersects $C_a$ with $x, y$ on the same $b$-side, then the intersection would either be one of the end points of $(x,y)$ or the entire line segment $[x,y]$. That is because an arc connecting the $a$-sides with the shortest length would not enter a region of $Q$ enclosed by a sub-arc of one of the $b$-sides and a line segment lying in $Q$. This ensures that $\Ce$ can be chosen so that there is no line segment with end points on the same $b$-side that intersects $\Ce$ and lies in $Q$ (except for its end points).

	We split the proof of Theorem \ref{thm_sa} in three Propositions. The first one asserts that the arc $\Ce$ exits every disk centered at points of $F$ and of radius $R$ in both directions (towards $z_1$ and $z_{N+1}$).

    \begin{Prop}\label{prop:step3}
		Let $w_0\in F$. If $\tilde{z}_1\in \Ce(z_1,w_0)$ with $\ell(\Ce(\tilde{z}_1,w_0)) \geq 15R$, then $\Ce(\tilde{z}_1, w_0)\cap (\C \setminus \overline{D(w_0,R)})\neq \emptyset$. Similarly, if $\tilde{z}_{N+1}\in \Ce(z_{N+1},w_0)$ with $\ell(\Ce(\tilde{z}_{N+1},w_0)) \geq 15R$, then $\Ce(\tilde{z}_{N+1}, w_0)\cap (\C \setminus \overline{D(w_0,R)})\neq \emptyset$.
	\end{Prop}
	
	\begin{proof}
		Let $w_0 \in F$ and $\tilde{z}_1 \in \Ce(z_1,w_0)$ with $\ell (\Ce(\tilde{z}_1,w_0)\geq 15R$. Assume towards a contradiction that $$\Ce(\tilde{z}_1,w_0) \subset \overline{D(w_0,R)}.$$
		
		Define the map $g:\overline{D(w_0,R)}\setminus\{w_0\}\rightarrow\partial D(w_0,R)$ by $$g(z):=\frac{R(z-w_0)}{|z-w_0|}+w_0$$ for all $z\in \overline{D(w_0,R)}\setminus\{w_0\}$. What $g$ does to a point $z$ of the closed punctured disk $\overline{D(w_0,R)}\setminus\{w_0\}$ is to map it to the point $g(z) \in \partial D(w_0,R)$ for which $(w_0,z) \subset (w_0,g(z))$.

		Let $w_1, w_2\in \Ce(\tilde{z}_1,w_0)$ with $\ell(\Ce(\tilde{z}_1,w_1))= 2R$ and $\ell(\Ce(w_0,w_2))= 2\epsilon +2R$. If $g(w_1)=g(w_2)$ then we can move along $\Ce$ and replace, for instance, $w_1$ with some $w'_1$ with $\ell(\Ce(\tilde{z}_1,w'_1)) \in [2R, 3R]$ and $g(w_1')\neq g(w_2)$, since $\Ce \subset \overline{D(w_0,R)}$ and it takes at most the length of $R$ to move to a different radius. As a result, we can assume $w_1, w_2$ do not lie on the same radius for the slightly "worse" scenario where
		\begin{equation}\label{w1}
			2R\leq\ell(\Ce(\tilde{z}_1,w_1))\leq3R,
		\end{equation}
		\begin{equation}\label{w2}
			\ell(\Ce(w_2,w_0))=2R+2\epsilon.
		\end{equation}
		
		Let $w \in \Ce(w_1,w_2)$, which implies that
		\begin{equation}\label{farz1}
		\ell(\Ce(\tilde{z}_1,w))\geq 2R,
		\end{equation}
		\begin{equation}\label{farw0}
		\ell(\Ce(w,w_0))\geq 2\epsilon+2R.
		\end{equation}
		
		If $(w_0,w)\subset \Ce$ then $\ell(\Ce(w_0,w))=|w-w_0|\leq R$ which contradicts \eqref{farw0}. Hence, $(w_0,w)\not\subset \Ce$. Additionally, since $\Ce$ is a union of finitely many line segments, there are points of $\Ce(w_1,w_2)$ not on the boundary $\partial D(w_0,R)$. So we can pick $w\notin \partial D(w_0,R)$.
		
		If $(w,w_0) \subset Q$ then we can replace $\Ce (w,w_0)$ by $[w,w_0]$ and, hence, we should have
		$$s_a+\epsilon-\ell(\Ce(w,w_0))+|w-w_0|\geq s_a,$$ so that
		$$\epsilon+|w-w_0| \geq \ell(\Ce(w,w_0)),$$
		which by \eqref{farw0} implies that
		$$\epsilon+ R \geq 2R+2 \epsilon,$$
		which is a contradiction. Therefore, $(w,w_0)\cap\partial Q\neq \emptyset$.
		
		Let $z_w\in \partial Q\cap (w,w_0)$ be the boundary point that is closest to $w$, i.e., with minimum $|z-w|$ among all boundary points $z$ of $Q$ on $(w,w_0)$. Since $(z_w,w)\subset Q$, if $z_w \in \partial_a Q$ then replacing part of $\Ce$ by the segment $[z_w, w]$ we see that either
		$$z_w\in \partial_{a_1}Q \Rightarrow s_a+\epsilon-\ell(\Ce(z_1,w))+|z_w-w|\geq s_a,$$
		or
		$$z_w\in \partial_{a_2}Q \Rightarrow s_a+\epsilon-\ell(\Ce(z_{N+1},w))+|z_w-w|\geq s_a.$$
		Recalling that $w\in \Ce(\tilde{z}_1,w_0)\subset \Ce(z_1,w_0)$, in both cases, because of \eqref{farz1} and \eqref{farendpnts} respectively, we would get that
		$$\epsilon\geq2R-|z_w-w| \geq R,$$
		which is a contradiction because we chose $\epsilon<10^{-3}r<R$. Hence, $z_w \in \partial_b Q$.
		
		Suppose $z_w \in \partial_{b_1}Q$. The proof is identical if $z_w \in \partial_{b_2}Q$.
		
			Suppose $[w,g(w)]\cap \partial Q \neq \emptyset$. Then there exists $B_w \in \partial Q \cap [w,g(w)]$ that is closest to $w$, i.e., with minimum $|B_w-w|$. If $B_w \in \partial_{a_1} Q$, then
			$$s_a+\epsilon-\ell(\Ce(z_1,w))+|B_w-w|\geq s_a,$$
			which by \eqref{farz1} implies that
			$$\epsilon-2R+R>0,$$ but that is a contradiction since $\epsilon<10^{-3}r<R$.
			
			Similarly, if $B_w\in \partial_{a_2}Q$ then $$s_a+\epsilon-\ell(\Ce(w,z_{N+1}))+|B_w-w|\geq s_a$$
			which by \eqref{farendpnts} leads to the contradiction $\epsilon-15R+R>0$. 
			
			Hence, $B_w \in \partial_bQ$. However, we assumed that there are no line segments with end points on the same $b$-side intersecting $\Ce$, so $B_w \in \partial_{b_2}Q$, which means that $[B_w,z_w]$ connects the two $b$-sides of $Q$, so 
			$$s_b\leq|B_w-z_w|\leq R= \frac{s_a}{100L} <s_b,$$ which is a contradiction.
			
			As a result, there is no $B_w \in \partial Q \cap [w,g(w)]$, which means that $g(w) \in Q$.
		
        Since $w$ was arbitrary, we have shown that the arc 
        $$A= \{ g(w): w\in \Ce(w_1,w_2) \}$$
	lies entirely in $Q$, and so does every segment $[w, g(w)]$ for $w \in \Ce (w_1, w_2)$. Hence, the arc 
        $$C'_\epsilon=\Ce(z_1,w_1)\cup A \cup \Ce(w_2, z_{N+1})\cup[w_1,g(w_1)]\cup [w_2,g(w_2)] \subset Q$$ joins the $a$-sides of $Q$ and needs to have length greater than or equal to $s_a$. But
    
    $$\ell(C'_\epsilon)= \ell(C_\epsilon(z_1,w_1))+\ell(A)+\ell(C_\epsilon(w_2,z_{N+1}))+|w_1-g(w_1)|+|w_2-g(w_2)|.$$
    So $\ell(\Ce')\geq s_a$ implies
    $$
    s_a+\epsilon-\ell(\Ce(w_1,w_2))+|w_1-g(w_1)|+|w_2-g(w_2)|+\ell(A)\geq s_a,
    $$
    which leads to
    $$
    \epsilon+2R+2\pi R\geq \ell(\Ce(w_1,w_2))=\ell(\Ce(\tilde{z}_1,w_0))-\ell(\Ce(\tilde{z}_1,w_1))-\ell(\Ce(w_2,w_0)).
    $$ But by \eqref{w1} and \eqref{w2} the above implies that $3\epsilon\geq R$, which is a contradiction. This finishes the proof for $\tilde{z}_1$. The proof is similar for $\tilde{z}_{N+1}$.
	\end{proof}

	The second Proposition asserts that every disk centered at points of $F'$ and of radius $R$ is split into two components by $\Ce$, only one of which may include boundary points of $Q$ outside the disk with same center of radius $R-\epsilon$ and within neighborhoods of the end points of the sub-arc of $\Ce$ lying in the disk (see Figure \ref{fig_bdry_in_nbh}).
	\begin{Prop}\label{prop:bdryannul}
		For every $w_0 \in F'$ there are points $\woa\in \Ce(z_1,w_0)\cap\partial D(w_0,R)$ and $\wob\in \Ce(w_0,z_{N+1})\cap\partial D(w_0,R)$ so that $D(w_0,R)\setminus \Ce(\woa,\wob)$ has exactly two connected components with closures $D^+$ and $D^-$. Moreover, at least one of $D^+\cap \partial Q$, $D^-\cap \partial Q$ is contained in $(D(\woa,2 \epsilon)\cup D(\wob,2 \epsilon))\setminus D(w_0,R-\epsilon)$.
	\end{Prop}
	
	\begin{figure}[H]
		\centering
		\includegraphics[width=0.7\textwidth]{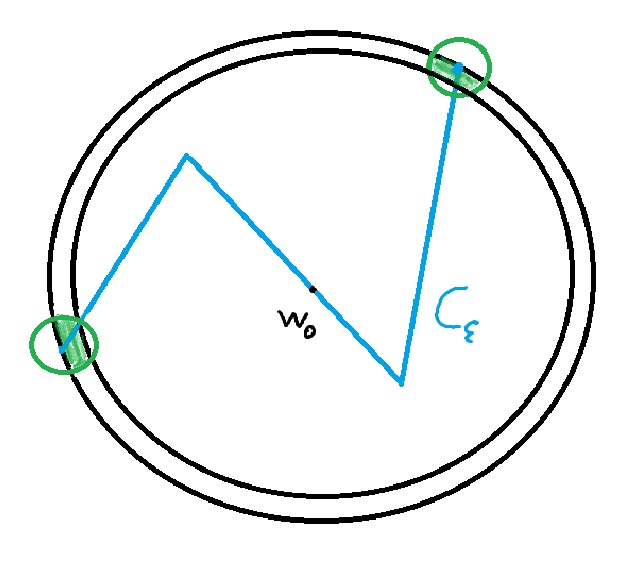}
		\caption{Boundary points of $Q$ can only lie in one component and the shaded areas.}
		\label{fig_bdry_in_nbh}
	\end{figure}
	
	\begin{proof}
		Let $w_0\in F'$. Note that since $F'\subset F$, by Proposition \ref{prop:step3} for $\tilde{z}_1=z_1$ there are two points $\woa\in \Ce(z_1,w_0)\cap\partial D(w_0,R)$ and $\wob\in \Ce(w_0,z_{N+1})\cap\partial D(w_0,R)$ with minimal $\ell(\Ce(\woa,w_0))$ and $\ell(\Ce(w_0,\wob))$ respectively so that $D(w_0,R)\setminus \Ce(\woa,\wob)$ has exactly two connected components, say $D^+$ and $D^-$. In addition, if $\ell(\Ce(\woa,w_0))\geq 15R$, by Proposition \ref{prop:step3} for $\tilde{z}_1=\woa$ we would get a contradiction regarding the minimality of $\ell(\Ce(\woa,w_0))$ among the points of $\Ce$ on $\partial D(w_0,R)$. Similarly for $\tilde{z}_{N+1}=\wob$, we conclude that $\ell(\Ce(\woa,w_0))< 15R$ and $\ell(\Ce(w_0,\wob))< 15R$. In the arguments that follow, note that $D^+\cup D^-$ may still contain points of $\Ce \setminus \Ce (\woa,\wob)$.

		Assume without loss of generality that $D^-\cap \partial Q$ is not contained in $(D(\woa,2\epsilon)\cup D(\wob,2\epsilon))\setminus D(w_0,R-\epsilon)$. We will show that $D^+\cap\partial Q\subset (D(\woa,2\epsilon)\cup D(\wob,2\epsilon))\setminus D(w_0,R-\epsilon)$. Let $$\tilde{z}_-\in \left(D^-\cap \partial Q\right) \setminus \left( (D(\woa,2\epsilon)\cup D(\wob,2\epsilon))\setminus D(w_0,R-\epsilon) \right).$$ Denote by $z_-$ the boundary point of $Q$ on $[\tilde{z}_-,w_0]$ that lies in $D^-$ and is closest to $w_0$. Then, denote by $w_-$ the point of $\Ce(\woa,\wob)$ that lies on $[z_-,w_0]$ and is closest to $z_-$ (so $w_-$ could be $w_0$). Hence, we end up with $z_- \in D^-\cap \partial Q$, $w_- \in \Ce(\woa,\wob)$ and $(z_-,w_-)\subset D^-\cap Q$.
		
		If $z_-\in \partial_{a_1}Q$, then 
		$$s_a+\epsilon-\ell(\Ce(z_1,w_-))+|z_--w_-|\geq s_a,$$ so that
		$$\epsilon+R\geq \ell(\Ce(z_1,w_-))=\ell (\Ce(z_1,w_0))\pm \ell(\Ce(w_-,w_0)),$$
		with $``+"$ if $w_- \in \Ce(w_0,\wob)$ and $``-"$ if $w_-\in \Ce(\woa,w_0)$. In either case, since $\ell (\Ce(w_-,w_0))\leq\ell(\Ce(\woa,w_0))<15R$ and $w_0 \in F'$, the right hand side of the above inequality is greater or equal to $16R+2\epsilon-15R=R+2\epsilon$, which leads to the contradictory inequality $\epsilon<0$. We get a similar contradiction if we assume that $z_- \in \partial_{a_2}Q$. Hence, $z_-\in \partial_b Q$.
		
		Suppose $z_- \in \partial_{b_1}Q$. If $D^+ \cap \partial Q = \emptyset$ then the statement of the Proposition follows. Suppose $D^+\cap \partial Q \neq \emptyset$ and consider an arbitrary point $\tilde{z}_+ \in D^+\cap \partial Q$.  Denote by $z_+$ the boundary point of $Q$ on $[\tilde{z}_+,w_0]$ that lies in $D^+$ and is closest to $w_0$. Then, denote by $w_+$ the point of $\Ce(\woa,\wob)$ that lies on $[z_+,w_0]$ and is closest to $z_+$. Hence, we end up with $z_+ \in D^+\cap \partial Q$, $w_+ \in \Ce(\woa,\wob)$ and $(z_+,w_+)\subset D^+\cap Q$.

		Similarly to $z_-$, we can show that $z_+\notin \partial_aQ$, so $z_+ \in \partial_bQ$. If $z_+\in \partial_{b_2}Q$ then we can join $\partial_{b_1} Q$ to $\partial_{b_2} Q$ in $Q$ by the path consisting of the line segments $[z_-,w_-]$ and $[z_+,w_+]$ and the arc $\Ce (w_-,w_+)$. Hence, by the definition of $s_b$ we have
		$$|z_- - w_-|+\ell(\Ce(w_-,w_+))+|z_+-w_+|\geq s_b$$
		But since $\ell(\Ce(\woa,w_0))< 15R$ and $\ell(\Ce(w_0,\wob))< 15R$, we have
		$$s_b \leq |z_- - w_-|+\ell(\Ce(w_-,w_+))+|z_+-w_+|< R+30R+R.$$
However, $R=\frac{s_a}{100L}$, so the above implies
		$$s_b<\frac{32 s_a}{100L}\leq \frac{32}{100} s_b,$$
		which is a contradiction. Hence, $z_+ \in \partial_{b_1}Q.$
		
		Suppose $[w_0,w_-] \nsubseteq Q$. In this case, let $z'_-$ be the boundary point  of $Q$ on $(w_0,w_-)$ that is closest to $w_-$. Similarly to $z_-$ this implies that $z'_- \in \partial_{b}Q$. But $z'_-$ cannot lie on $\partial_{b_1}Q$ because of our assumption on line segments with end points on the same $b$-side not intersecting $\Ce$. Thus, $z'_- \in \partial_{b_2}Q$, which is a contradiction because $(z_-,z'_-)\subset Q$ and $|z_- -z'_-|\leq R<s_b$. Following the same argument for $[w_0,w_+]$, we get that $[w_0,w_-], [w_0,w_+]\subset Q$.

		 Assume towards a contradiction that there are no points of $\Ce$ on $(z_-,w_-)\cup (z_+,w_+)$. Denote by $T$ the union of the closed line segments of $\Ce$ that intersect $w_-$ and $w_+$ and denote by $|T|$ the number of said line segments. Note that $|T|\in \{2,3,4\}$ based on whether $w_-, w_+$ are end points of some $[z_i,z_{i+1}]$ or not. Set $z_1', z_{N+1}'\in \Ce$ with $\ell(\Ce(z_1,z_1'))=\ell(\Ce(z_{N+1}',z_{N+1}))=R/2$ and pick some tiny positive $\tilde{\epsilon}<10^{-5}\epsilon$ so that $D(w,\tilde{\epsilon})\subset Q$  for all $ w\in \Ce(z_1',z_{N+1}')$. Let 
			\begin{align*}
                N_{\Ce} =\{ z\in D(w,\tilde{\epsilon}): w\in \Ce(z_1',z_{N+1}') \} \cup &\\
                \cup \{ z\in D(w,\tilde{\epsilon}) \cap Q: w\in \Ce(z_1,z_1')&\cup\Ce(z_{N+1}',z_{N+1}) \}
            \end{align*}
be a neighborhood of $\Ce\setminus \{z_1, z_{N+1}\}$ inside $Q$ so that 
		\begin{itemize}
			\item $(z_-,w_-)$ intersects the boundary of only one connected component of $N_{\Ce}\setminus\Ce$, for instance by taking $\tilde{\epsilon}<\frac{\dist((z_-,w_-),\, \Ce\setminus T)}{2}$,
			\item $(z_+,w_+)$ intersects the boundary of only one connected component of $N_{\Ce}\setminus\Ce$, for instance by taking $\tilde{\epsilon}<\frac{\dist((z_+,w_+),\, \Ce\setminus T)}{2}$.
		\end{itemize}

		\begin{figure}[H]
			\centering
			\includegraphics[width=0.85\textwidth]{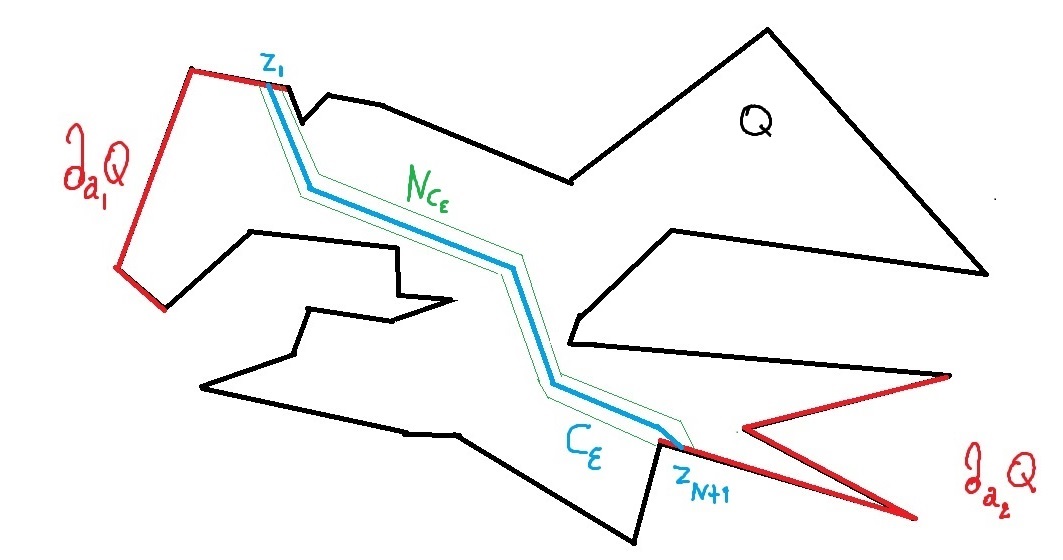}
			\caption{An example of $N_{\Ce}$.}
			\label{fig_def_NCe}
		\end{figure}
	
		As a result, $(z_-,w_-)$ and $(z_+,w_+)$ intersect $\partial N_{\Ce}$ at unique points $n_-$ and $n_+$, respectively. If both $(n_-,w_-)$, $(n_+,w_+)$ lie in the same connected component of $N_{\Ce}\setminus \Ce$, then we could find a arc $C_{\pm}$ that connects $n_-$ with $n_+$ inside the closure of the same component of $N_{\Ce}\setminus \Ce$. But for $\tilde{\epsilon}$ small enough, this arc can be chosen so that it lies inside $D(w_0,R)$ and does not intersect $\Ce$, implying that the arc $(z_-,n_-)\cup C_{\pm} \cup (n_+,z_+)$ connects $z_-$ with $z_+$ without intersecting $\Ce$.
	
		\begin{figure}
			\centering
			\includegraphics[width=0.7\textwidth]{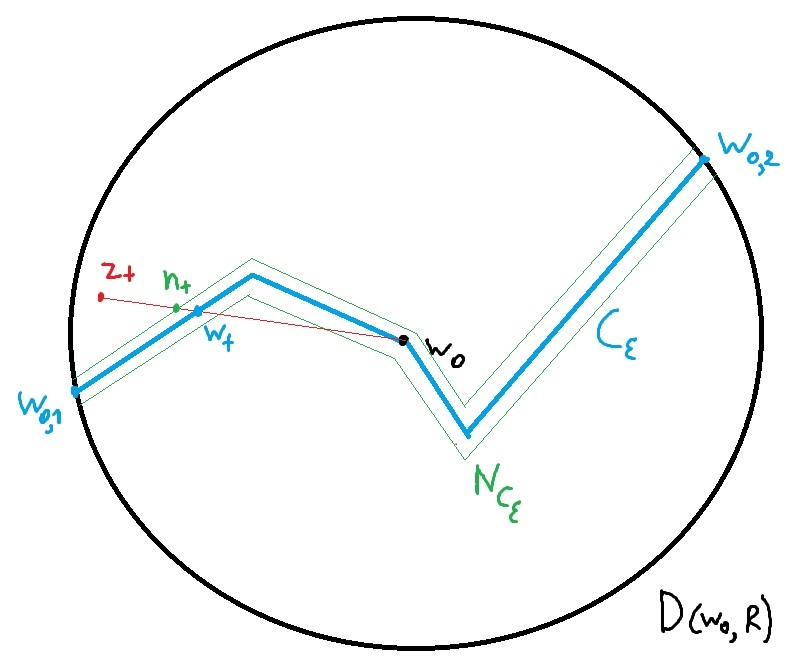}
			\caption{Showing how thin $N_{\Ce}$ is chosen to be, even inside $D(w_0, R)$ and compared to $|z_+-w_+|$, $|z_- - w_-|$.}
			\label{fig_thinNCe}
		\end{figure}

		But this contradicts the fact that $z_-\in D^-$ and $z_+\in D^+$, which are different connected components of $D(w_0,R)\setminus \Ce(\woa,\wob)$. Hence, $(n_-,w_-)$ and $(n_+,w_+)$ lie in different components of $N_{\Ce}\setminus \Ce$. If $Q_1$ and $Q_2$ are the two connected components of $Q\setminus \Ce$ that contain $\partial_{b_1}Q$ and $\partial_{b_2}Q$ on their boundary, respectively, then the two connected components of $N_{\Ce}\setminus \Ce$ would lie in $Q_1$ and $Q_2$, say the one including $(n_-, w_-)$ lies in $Q_1$ and the other in $Q_2$ and the proof is identical if it is the other way around. To see why this is not possible, unless one of $(z_-,w_-)$, $(z_+,w_+)$ intersects $\Ce$, it helps to map the quadrilateral $Q$ onto a rectangle $Rec(Q)$ using a conformal map $\phi$ so that $\phi (\partial_{a_1}Q)=(0,M)$, $\phi (\partial_{b_1}Q)=(M,M+i)$,  $\phi (\partial_{a_2}Q)=(i,M+i)$, $\phi (\partial_{b_2}Q)=(0,i)$, where $M=\text{Mod}(Q)$.
		 
		\begin{figure}[H]
			\centering
			\includegraphics[width=0.9\textwidth]{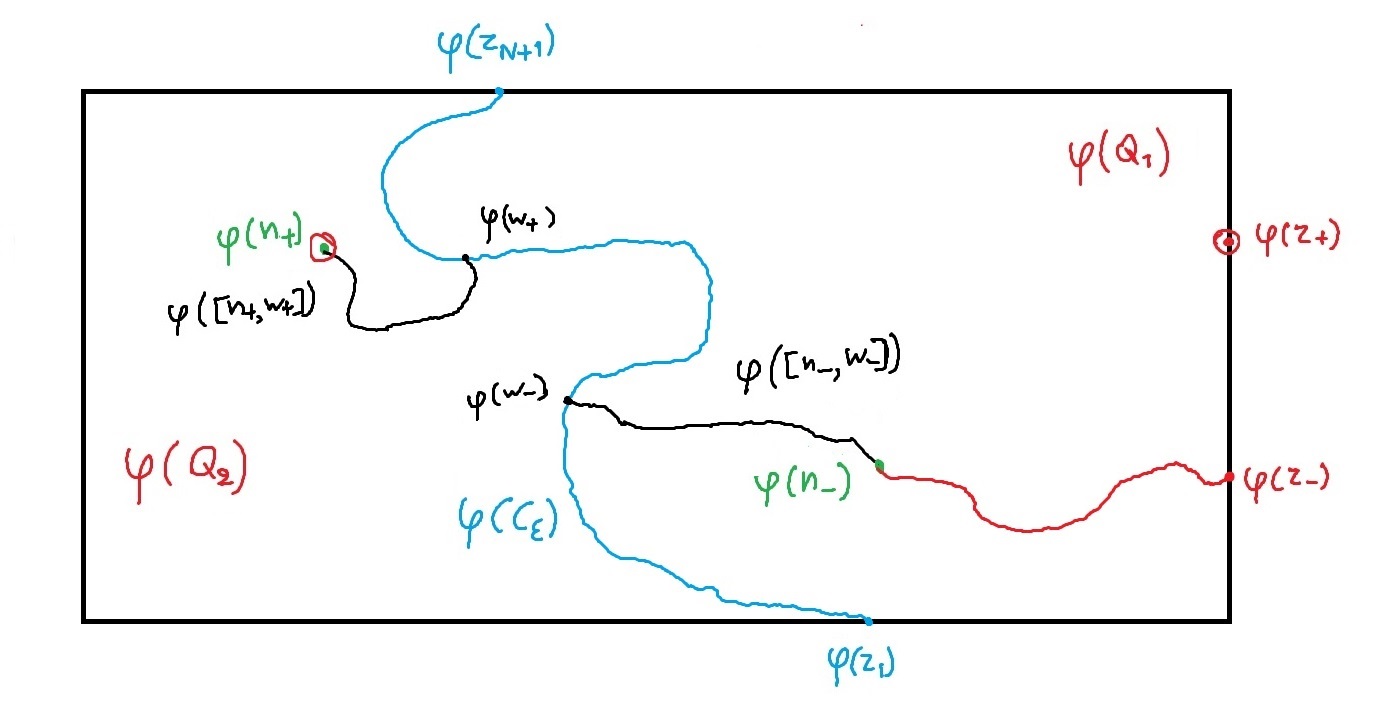}
			\caption{In case $[z_+, n+]\subset Q_2$, there is no way to connect $\phi(z_+)$ to $\phi(n_+)$ without crossing $\phi (\Ce)$ or the boundary of $Rec(Q)$. Similarly in the case $[z_+, n+]\subset Q_1$ for $\phi(z_-)$ and $\phi(n_-)$.}
			\label{fig_rec_sameb-side}
		\end{figure}
	
		What we have shown is that there is a point on the right vertical side of $Rec(Q)$, specifically $ \phi(z_+)$, which can be connected to a point of $\phi (Q_2)$, specifically $\phi(n_+)$, by a arc inside $Rec(Q)$ that does not intersect $\phi(\Ce)$. But since such a arc would start on the right vertical side and $\phi(n_+)\in \phi(Q_2)$, it intersects both $\phi(Q_1)$ and $\phi(Q_2)$, which can only be achieved either by crossing $\phi(\Ce)$ or the boundary of $Rec(Q)$. However, our hypothesis is that none of these two cases occurs, which leads to a contradiction.

        We reached the above contradiction because we assumed that there are no points of $\Ce$ on $(z_-,w_-)\cup (z_+,w_+)$. As a result, $\Ce$ intersects at least one of $(z_-,w_-)$, $(z_+,w_+)$. Suppose it intersects $(z_+,w_+)$. Note that by the definition of $w_+$, there are no points of $\Ce(\woa,\wob)$ on $(z_+,w_+)$. Thus, there is some $w_+' \in \Ce \setminus\Ce(\woa,\wob)$ that lies on $(z_+,w_+)$. But we showed that $(z_+,w_0) \subset Q$, which implies that $(w_+',w_0)\subset Q$. By minimality of $s_a$ we have
		 $$
		 s_a+\epsilon-\ell(\Ce(w_0,w_+'))+|w_+' -w_0|\geq s_a,
		 $$
		 so that
	\begin{equation}\label{pushpnt+}
        |w_+' -w_0|\geq \ell(\Ce(w_0,w_+'))-\epsilon.
	\end{equation}

        Similarly, if $\Ce$ intersects $(z_-, w_-)$ as well, we get
        \begin{equation}\label{pushpnt-}
            |w_-' -w_0|\geq \ell(\Ce(w_0,w_-'))-\epsilon.
        \end{equation}

        Recall that $w_-' \in \Ce(z_1,\woa)\cup \Ce(\wob,z_{N+1})$. If $w_-' \in \Ce(z_1,\woa)$ then
        $$
        \ell(\Ce(w_0,w_-')) = \ell(\Ce(w_0,\woa))+ \ell(\Ce(\woa,w_-'))\geq R+|\woa-w_-'|,
        $$ which combined with \eqref{pushpnt-} and $|w_-'-w_0|\leq R$ implies that
        $$
        \epsilon\geq |\woa-w_-'|.
        $$ But $|\tilde{z}_--w_-'|= |\tilde{z}_- -w_0|-|w_-'-w_0|\leq R-\ell(\Ce(w_0,w_-'))+\epsilon\leq \epsilon$. As a result,
        \begin{equation}\label{eq:2epsilon_needed}
        |\tilde{z}_- -\woa|\leq|\tilde{z}_--w_-'|+|\woa-w_-'| \leq 2\epsilon,
        \end{equation} which contradicts the choice of $\tilde{z}_-\in (D^-\cap \partial Q) \setminus (D(\woa,2\epsilon)\cup D(\wob,2\epsilon))$. Similarly, $w_-' \in \Ce(\wob, z_{N+1})$ implies that $|\tilde{z}_- -\wob|\leq 2\epsilon$, which is also a contradiction.

        Hence, $\Ce$ can only intersect $(z_+,w_+)$. Suppose $w_+'\in \Ce(z_1, \woa)$. Since 
        $$
        |\tilde{z}_+-\woa|\leq |\tilde{z}_+ -w_+'| + |\woa-w_+'|,
        $$ and because the right hand side equals $|\tilde{z}_+ - w_0|-|w_+' -w_0|+|\woa-w_+'|$, we get by \eqref{pushpnt+} and $|\tilde{z}_+-w_0|\leq R$ that
        \begin{equation}\label{eq:w_+'indisk}
        |\tilde{z}_+-\woa|\leq R-\ell(\Ce(w_0,w_+'))+\epsilon+|\woa-w_+'|.
        \end{equation} Because $w_-' \in \Ce(z_1,\woa)$ we get
        $$
        \ell(\Ce(w_0,w_+')) = \ell(\Ce(w_0,\woa))+ \ell(\Ce(\woa,w_+'))\geq R+|\woa-w_+'|,
        $$ which combined with \eqref{eq:w_+'indisk} implies that
        $$
        |\tilde{z}_+-\woa|\leq  \epsilon.
        $$ Similarly, if $w_+'\in \Ce(\wob, z_{N+1})$ we can show that $|\tilde{z}_+-\wob|\leq  \epsilon.$ Hence, $|\tilde{z}_+ -\woa|\leq \epsilon$ or $|\tilde{z}_+ -\wob|\leq \epsilon$, each of which implies that $|\tilde{z}_+-w_0|\geq R-\epsilon$.
        
        Note that the assumption $\tilde{z}_-\in (D^- \cap \partial Q)\setminus (D(\woa, 2\epsilon)\cup D(\wob,2\epsilon))$ was necessary because of \eqref{eq:2epsilon_needed}. Since $\tilde{z}_+\in D^+\cap \partial Q$ was arbitrary, the proof is complete. 
        
        
\end{proof}

\begin{Rem}\label{re:D+in_nbhd_of_ls}
    For the rest of the paper we assume without loss of generality that $D^+ \cap \partial Q \subset (D(\woa,2 \epsilon)\cup D(\wob,2 \epsilon))\setminus D(w_0,R-\epsilon)$. Observe that if for the boundary point $\tilde{z}_+\in D^+ \cap \partial Q$ the corresponding $w_+'$ lies in $\Ce(z_1,\woa)$, then by \eqref{pushpnt+} we have $\ell(\Ce(\woa,w_0))<R+2\epsilon$, and similarly if $w_+'\in \Ce(\wob,z_{N+1})$ then $\ell(\Ce(\wob,w_0))<R+2\epsilon$. This means that at least one of $\Ce(\woa,w_0)$, $\Ce(w_0,\wob)$ cannot deviate much from being the line segment $[\woa,w_0]$, $[w_0,\wob]$ respectively. In other words, at least one of them lies in a $3\epsilon$-neighborhood of the respective line segment. In the case $D^+\cap \partial Q \neq \emptyset$ assume for what follows that
    $$
    \Ce(\woa,w_0)\subset N_{3\epsilon}=\{z\in D(w, 3\epsilon): w\in [\woa,w_0] \}
    $$ without loss of generality, since in case $\Ce(w_0, \wob)\subset N_{3\epsilon}'=\{z\in D(w, 3\epsilon): w\in [w_0, \wob] \}$ the proof is identical.
    
\end{Rem}
The following Proposition finishes the proof of Theorem \ref{thm_sa} by placing a disk of radius $r$ inside the part of the ``good" component $D^+$ that contains no points of $\partial Q$.
\begin{Prop}\label{prop disk_place}
    Let $D^+$, $D^-$ be as in Proposition 2 and Remark \ref{re:D+in_nbhd_of_ls}. Then there is $w_0'\in D^+$ such that $$D(w_0',r)\subset D^+ \cap D(w_0,R-\epsilon)  \subset Q.$$
\end{Prop}

\begin{proof}
    Suppose $D^+\cap \partial Q \neq \emptyset$. Denote by $n_0$ and $n_0'$ the points where $\partial N_{3\epsilon}$ intersects the line $E$ perpendicular to $[\woa,w_0]$ at $w_{2r}\in[\woa,w_0]$ with $|\woa-w_{2r}|=2r$. Set $y_{2r}$ to be the point of $\Ce(\woa,w_0)$ that lies on $[w_{2r},n_0]$ and is closest to $n_0$ and similarly set $y_{2r}'\in \Ce(\woa,w_0)$ to be the point on $[w_{2r},n_0']$ that is closest to $n_0'$. Moreover, set $y_w$ and $y_w'$ to be the two points on $E$ with $|y_w-w_{2r}|=|y_w'-w_{2r}|=r+3\epsilon$, where  $n_0\in [y_w, w_{2r}]$ and $n_0'\in [y_w', w_{2r}]$. Assume towards a contradiction that both $y_w, y_w' \in D^-$. Then at least one of the line segments $[y_{2r}, y_w]$, $[y_{2r}', y_w']$ intersects $\Ce(w_0,\wob)$. Assume without loss of generality that $[y_{2r}, y_w]\cap \Ce(w_0,\wob) \neq \emptyset$ and the proof is identical in the other case. Let $\tilde{y}_w$ be the point on $[y_{2r}, y_w]\cap \Ce(w_0,\wob)$ that is closest to $y_{2r}$. Then $[\tilde{y}_w, y_{2r}] \subset D^+$ and since it does not lie in $D(\woa, 2\epsilon)\cup D(\wob, 2\epsilon)$, by Proposition \ref{prop:bdryannul} and Remark \ref{re:D+in_nbhd_of_ls} we have that $[\tilde{y}_w, y_{2r}] \subset Q$. However, the Jordan arc $(\Ce \setminus \Ce(y_{2r}, \tilde{y}_w))\cup [\tilde{y}_w, y_{2r}]$ lies in $Q$ apart from its end points and connects its $a$-sides with length less or equal to
    $$
    s_a+\epsilon - \ell(\Ce(y_{2r}, \tilde{y}_w))+|\tilde{y}_w- y_{2r}|\leq s_a+\epsilon-8R/10-8R/10+r+3\epsilon,
    $$ in which the right hand side, by choice of $\epsilon$ and $r=R/10$, is strictly less than $s_a$ and leads to a contradiction. Thus, at least one of $y_w$, $y_w'$ lies in $D^+$, which we denote by $w_0'$. Assume without loss of generality that $y_w=w_0'$ and the proof is identical in the other case.
    
    Then the disk $D(w_0',r)=D(w_0',R/10)$ with $w_0' \in D^+$  is tangent to $\partial N_{3\epsilon}$ at $n_0$. We claim that $D(w_0',r)\subset D^+$. This would finish the proof, because by the choice of $w_0'$ it is easy to see that $D(w_0',r) \subset D(w_0, R-\epsilon)$. 
    \begin{figure}
    	\centering
    	\includegraphics[width=0.9\textwidth]{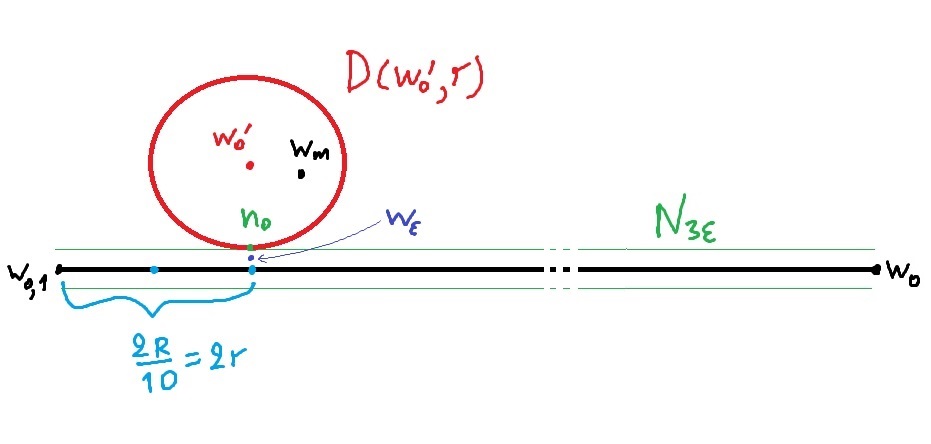}
    	\caption{The disk $D(w_0',r)$ tangent to $N_{3\epsilon}$.}
    	\label{fig_D+_hasbdry}
    \end{figure}To show this, it would be enough to show that no point of $\Ce(w_0, \wob)$ can lie in the interior of $D(w_0',r)$, something we know already for $\Ce(w_0, \woa)$ because it lies in $N_{3\epsilon}$. Assume towards a contradiction that this is not the case, and let $w_m$ be the point of $\Ce(w_0, \wob)$ inside $D(w_0',r)$ that is closest to $n_0$. The way to reach a contradiction is to show that this leads to an arc connecting the $a$-sides with length less than $s_a$. Let $w_\epsilon$ be the point of $\Ce(\woa,w_0)$ that lies on the line defined by $(w_0',n_0)$ and is closest possible to $n_0$. Then the arc
    $$
    \tilde{C}_\epsilon = \Ce(z_1, w_\epsilon) \cup [w_\epsilon, n_0]\cup [n_0,w_m]\cup \Ce(w_m, z_{N+1})
    $$ lies in $Q$, joins the $a$-sides of $Q$, and has length that must be at least $s_a$. But
    $$
    \ell(\tilde{C}_\epsilon)=s_a+\epsilon-\ell(\Ce(w_\epsilon,w_0))-\ell(\Ce(w_0,w_m))+|w_\epsilon-n_0|+|n_0-w_m|
    $$ and $|w_\epsilon-n_0|\leq 6\epsilon$, $|n_0-w_m|\leq 2r$. Hence, $\ell(\tilde{C}_\epsilon)\geq s_a$ implies that
    $$
    \epsilon+6\epsilon+2r\geq \ell(\Ce(w_\epsilon,w_0))+\ell(\Ce(w_0,w_m)).
    $$ But by the definition of $D(w_0',r)$ and $w_\epsilon$, the points $w_m$ and $w_\epsilon$ cannot lie in $D(w_0, R/2)$. Hence, since $\ell(\Ce(w_\epsilon,w_0))\geq |w_\epsilon-w_0|$ and $\ell(\Ce(w_m,w_0))\geq |w_m-w_0|$, we get
    $$
    7\epsilon+2r\geq R,
    $$ and recalling $r=R/ 10$ the above yields
    $$
    \epsilon\geq 4R/35,
    $$ which is a contradiction. As a result, $D(w_0',r) \subset D^+ \cap D(w_0, R-\epsilon) \subset Q$.

    Suppose $D^+ \cap \partial Q = \emptyset$. Let $\theta$ be the angular measure of the arc $A^+=\partial D^+ \cap \partial D(w_0,R)$. Let $d_1, \dots , d_7 \in A^+$ be such that the angle of the sub-arc of $A^+$ connecting $d_j$ with $d_{j+1}$ has measure $\theta R/8$ for all $0\leq j \leq 7$, where $d_0=\woa$ and $d_8= \wob$. For every $j$ with $1\leq j \leq 7$ denote by $D_j$ the disk $D(w_{d_j},r)\subset D(w_0,R)$ that is tangent to $\partial D(w_0,R)$ at the point $d_j$. If there is $j$ for which $D_j\cap \Ce(\woa,\wob)= \emptyset$ then $D_j\subset D^+\subset Q$ and the proof is complete.
        
    Assume towards a contradiction that all $D_j$ intersect  $\Ce(\woa,\wob)$ and denote by $c_j$ a point of $\Ce(\woa,\wob)$ in $D_j$ with minimal distance $|d_j-c_j|$. Then $[c_j,d_j]\subset D^+ \subset Q$ and $|c_j-d_j|\leq 2r=R/5$.
        
    \begin{figure}[H]
    	\centering
    	\includegraphics[width=0.8\textwidth]{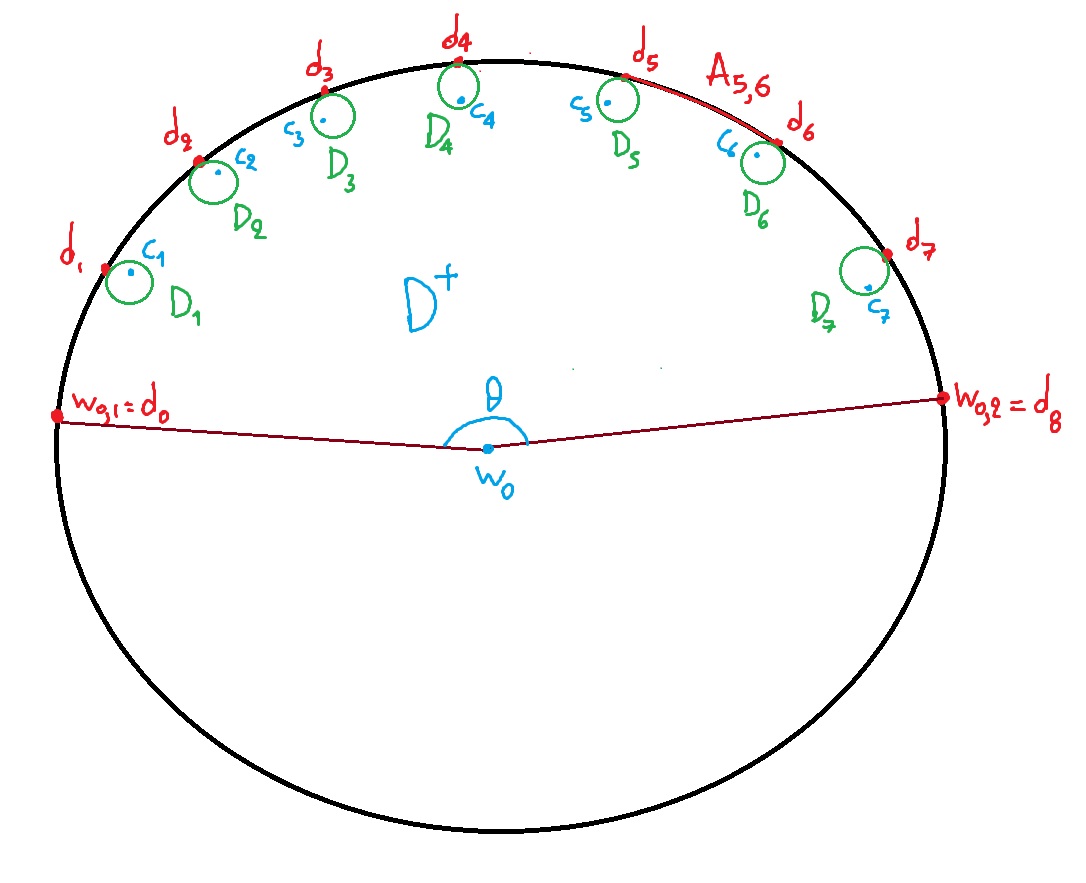}
    	\caption{The disks $D_j$ tangent to $\partial D(w_0,R)$ from the inside.}
    	\label{fig_D+_no_bdry}
    \end{figure}
        
    We will first show that all $c_j$ need to lie in the same component of $\Ce(\woa,\wob)\setminus \{w_0\}$. Indeed, suppose that there is $j\in [1,6]$ such that $c_j\in \Ce(\woa,w_0)$ and $c_{j+1} \in \Ce(w_0,\wob)$ (the proof is similar if the roles of $c_j$ and $c_{j+1}$ are reversed). If $A_{j,j+1}$ is the sub-arc of $A^+$ connecting $d_j$ and $d_{j+1}$, then the arc
    $$
     \Ce'=\Ce(z_1,c_j)\cup (c_j,d_j)\cup A_{j,j+1}\cup (d_{j+1},c_{j+1})\cup \Ce(c_{j+1},z_{N+1})
    $$ connects $\partial_{a_1} Q$ and $\partial_{a_2}Q$ inside $Q$. Hence, $\ell(\Ce')\geq s_a$, which implies        
    $$
        s_a+\epsilon-\ell(\Ce(c_j,w_0))-\ell(\Ce(w_0,c_{j+1}))+|c_j-d_j|+|c_{j+1}-d_{j+1}|+\ell(A_{j, j+1})\geq s_a.
    $$
    But then
    $$
        \epsilon+2r+2r+\theta R/8\geq \ell(\Ce(c_j,w_0))+\ell(\Ce(w_0,c_{j+1}))\geq 8R/10+8R/10,
    $$ which by $r=R/10$ and $\theta< 2\pi$ implies        
    $$
       \epsilon\geq 12R/10-\pi R/4 >2R/10,
    $$ which is a contradiction.
        
    As a result, all $c_j$'s lie in the same component of $\Ce(\woa,\wob)\setminus \{ w_0\}$, for all $j\in [1,7]$ . Assume that $c_j\in \Ce (\woa,w_0)$ for all $j\in [1,7]$, since the proof is identical in the case where all $c_j$ lie in $\Ce(w_0, \wob)$ instead. Then the arc
    $$
        \Ce''=\Ce(z_1,c_7)\cup(c_7,d_7)\cup A_{7,8}\cup \Ce(\wob,z_{N+1})
    $$ connects $\partial_{a_1} Q$ and $\partial_{a_2} Q$ inside $Q$. Similarly to $\Ce'$, this implies
    $$
        \epsilon +2R/10+\theta R/8\geq 8R/10+8R/10,
    $$ which gives the contradiction $\epsilon>14R/10-\pi R/4>4R/10$ and completes the proof.

\end{proof}

\section{Final remarks}
A natural question to ask is whether some kind of converse to Theorem~\ref{thm_main_Quad} could potentially hold. For instance, for a fixed sufficiently small $\epsilon>0$ and a fixed $\delta>0$, and for the collection $\mathcal{Q}_\delta$ of quadrilaterals $Q$ for which for every $w_0\in F'$ as in \eqref{furtherendpnts} there is a disk of radius $r=\delta \max\{s_a(Q),s_b(Q)\}$ within $D(w_0,10r)$ that lies in $Q$, is there a global bound on the modulus $M(Q)$ for all $Q\in \mathcal{Q}_\delta$ that depends only on $\delta$? Such a converse cannot be true, as demonstrated in Figures \ref{fig:countereg} and \ref{fig:Zoomcountereg}, even under the stronger assumption that there is a Jordan arc of length $s_a$ and every disk of radius $r=\delta \max\{s_a(Q),s_b(Q)\}$ centered on said arc lies entirely in $Q$.
\begin{figure}[H]
	\centering
	\includegraphics[height=0.2\textheight, width=0.9\textwidth]{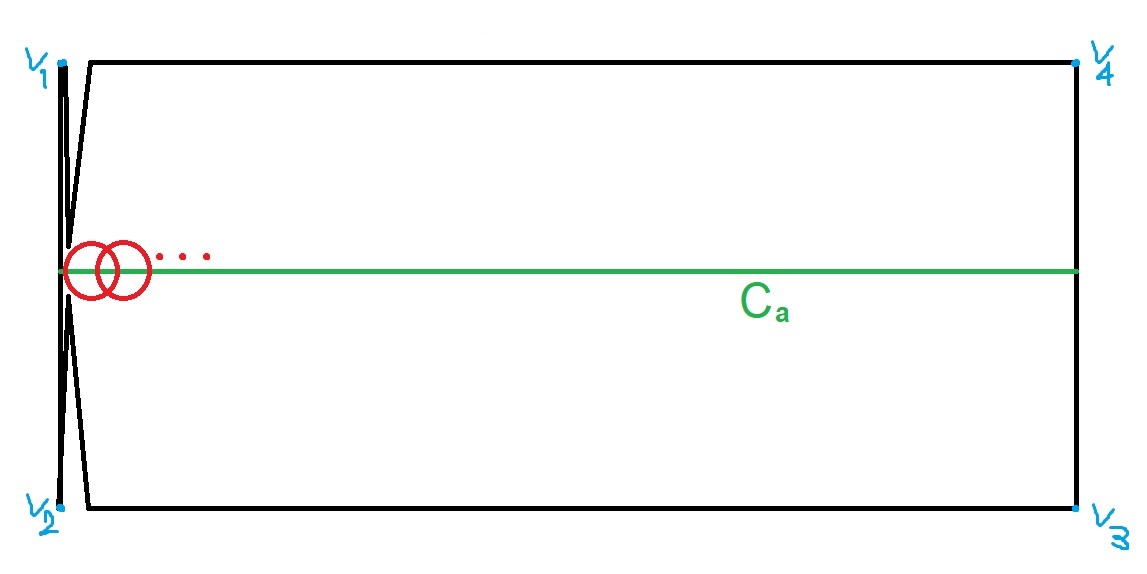}
	\caption{No matter how small the red disks are, the two pointy parts of the $b$-sides can be as close as needed to make the modulus too large.}
	\label{fig:countereg}
\end{figure}

\vspace{2cm}
\begin{figure}[H]
	\centering
	\includegraphics[width=0.4\textwidth, height=0.4\textheight]{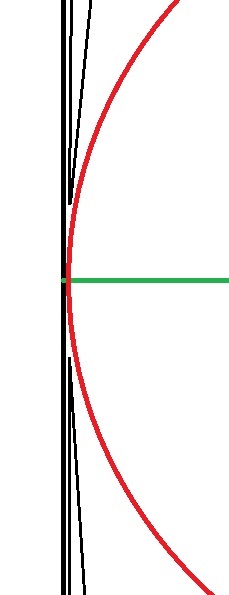}
	\caption{Zooming in around the left end point of $C_a$ from Figure \ref{fig:countereg}.}
	\label{fig:Zoomcountereg}
\end{figure}

Despite Theorem \ref{thm_main_Quad} not being a complete characterization of such collections of quadrilaterals with globally bounded modulus, it might still contribute to characterizations of planar quasiconformal maps. For instance, a result of similar geometric flavor was used in \cite{AHH} to prove that if a homeomorphism maps all equilateral triangles onto topological triangles whose vertices satisfy a condition related to bounded distortion, then it has to be quasiconformal. Other a priori weaker properties that ended up being enough to define quasiconformality have been given by Hinkkanen \cite{Ai}, Aseev \cite{As}, and Ackermann \cite{Ack}.

It is also important to point out that Propositions \ref{prop:step3}, \ref{prop:bdryannul}, and \ref{prop disk_place} provide interesting properties regarding the boundary points within components of certain disks in a quadrilateral, as well as an approximate location of the desired disk of Theorem~\ref{thm_main_Quad} lying inside the quadrilateral. Indeed, we prove that if $w_0$ is an arbitrary point of the set $F'$ defined by (\ref{furtherendpnts}) (that is, the set of points on the arc joining the $a-$sides of the quadrilateral not too close to the end points of the arc), then the disk $D(w_0,R)$ contains a disk of radius $r$ contained in the quadrilateral $Q$, where $R$ and $r$ are as just above (\ref{farendpnts}). 

A lot of the arguments in our proofs would be simplified if the arc $C_\eps$ defined in the second paragraph of Section 4 had length equal to $s_a$. Namely, if $C_a$ could be chosen to lie in the interior of $Q$ with end points not on the vertices of $Q$, in which case $C_\eps=C_a$. It was pointed out already after the definition of $C_\eps$ that the two arcs need not be the same, as it is depicted in Figure \ref{RefFig}. This raises the question (also proposed to us by the anonymous referee), whether there is a characterization for quadrilaterals $Q$ in $\mathcal{Q}_{\text{ls}}$ for which $C_a$ can be selected to lie in the interior of $Q$. To the best of our knowledge, this is an open problem with interest on its own, which would require a closer analysis of properties of ``linear" quadrilaterals.

\begin{figure}[H]
	\centering
	\includegraphics[width=0.9\textwidth]{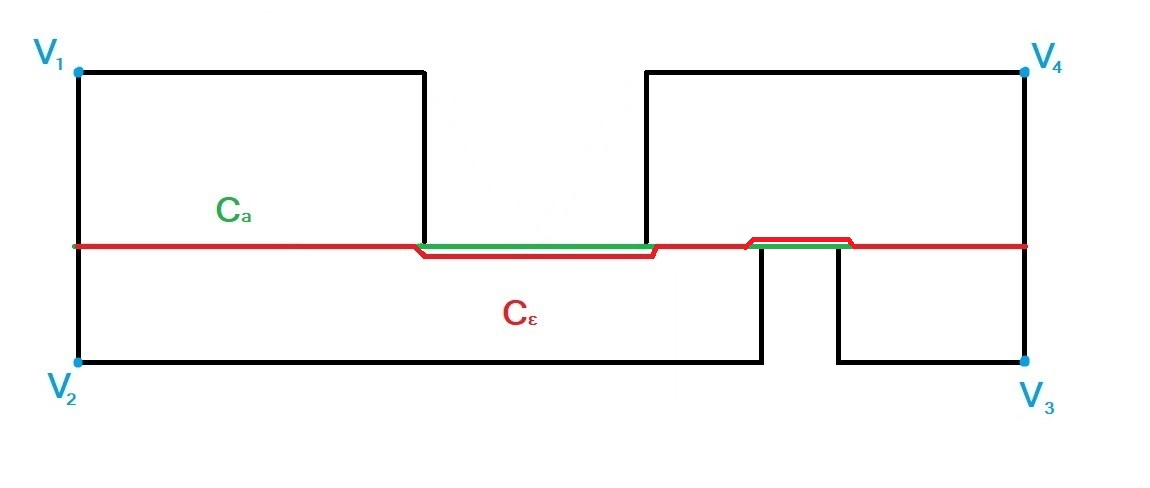}
	\caption{Example where $C_a$ (in green) and $C_\epsilon$ (in red) differ.}
	\label{RefFig}
\end{figure}

\bibliographystyle{acm}

\end{document}